\documentclass[11pt]{amsart}

\usepackage{amssymb}
\usepackage{bm}
\usepackage[centertags]{amsmath}
\usepackage{amsfonts}
\usepackage{amsthm}
\usepackage{mathrsfs}

\usepackage{graphicx}

\theoremstyle{definition}
\newtheorem{thm}{Theorem}[section]
\newtheorem{dfn}[thm]{Definition}
\newtheorem{lem}[thm]{Lemma}
\newtheorem{prp}[thm]{Proposition}
\newtheorem{cor}[thm]{Corollary}
\newtheorem{rmk}[thm]{Remark}
\newtheorem{rmks}[thm]{Remarks}
\newtheorem{prb}{Problem}

\newtheorem*{thm*}{Theorem}

\newtheorem*{cor*}{Corollary}

\newtheorem*{prp*}{Proposition}

\newtheorem*{rmk*}{Remark}

\newtheorem*{ntt}{Notation}

\newcommand{\inn}{\in\mathbb{N}}

\newcommand{\e}{\varepsilon}
\newcommand{\al}{\alpha}
\newcommand{\de}{\delta}
\newcommand{\la}{\lambda}

\newcommand{\be}{\beta}
\newcommand{\adx}{\al\big(\{x_k\}_k\big)}
\newcommand{\bdx}{\be\big(\{x_k\}_k\big)}
\newcommand{\ady}{\al\big(\{y_k\}_k\big)}
\newcommand{\bdy}{\be\big(\{y_k\}_k\big)}
\newcommand{\adz}{\al\big(\{z_k\}_k\big)}
\newcommand{\bdz}{\be\big(\{z_k\}_k\big)}
\newcommand{\adw}{\al\big(\{w_k\}_k\big)}
\newcommand{\bdw}{\be\big(\{w_k\}_k\big)}
\newcommand{\Sn}{\mathcal{S}_{n}}
\newcommand{\Sj}{\mathcal{S}_j}
\newcommand{\X}{\mathfrak{X}_{_{^\text{usm}}}}
\newcommand{\plusminus}{\substack{+\\[-2pt]-}}
\newcommand{\bcrossx}{b\otimes\{x_k\}_k}

\newcommand{\Bcrossx}{\mathcal{B}\otimes\{x_k\}_k}
\newcommand{\Bcrossy}{\mathcal{B}\otimes\{y_k\}_k}
\newcommand{\Bcrossz}{\mathcal{B}\otimes\{z_k\}_k}

\DeclareMathOperator{\supp}{supp}

\DeclareMathOperator{\ran}{ran}

\DeclareMathOperator{\sgn}{sgn}

\long\def\symbolfootnote[#1]#2{\begingroup%
\def\thefootnote{\fnsymbol{footnote}}\footnote[#1]{#2}\endgroup}

\begin{document}

\title[Rich
spreading model structure]{A hereditarily indecomposable Banach
space with rich spreading model structure}
\dedicatory{Dedicated to the memory of Joram Lindenstrauss}

\author[S.A. Argyros, P. Motakis]{Spiros A. Argyros, Pavlos
Motakis}
\address{National Technical University of Athens, Faculty of Applied Sciences,
Department of Mathematics, Zografou Campus, 157 80, Athens,
Greece} \email{sargyros@math.ntua.gr, pmotakis@central.ntua.gr}

\symbolfootnote[0]{\textit{2010 Mathematics Subject
Classification:} Primary 46B03, 46B06, 46B25, 46B45}

\symbolfootnote[0]{\textit{Key words:} Spreading models, Strictly
singular operators, Reflexive spaces, Hereditarily indecomposable
spaces}

\symbolfootnote[0]{Research supported by API$\Sigma$TEIA program.}

\maketitle
\begin{abstract}
We present a reflexive Banach space $\X$ which is Hereditarily
Indecomposable and satisfies  the following properties. In every
subspace $Y$ of $\X$ there exists a weakly null normalized
sequence $\{y_n\}_n$, such that every subsymmetric sequence
$\{z_n\}_n$ is isomorphically generated as a spreading model of a
subsequence of $\{y_n\}_n$. Also, in every block subspace $Y$ of
$\X$ there exists a seminormalized block sequence $\{z_n\}$ and
$T:\X\rightarrow\X$ an isomorphism such that for every $n\inn$
$T(z_{2n-1}) = z_{2n}$. Thus the space is an example of an HI
space which is not tight by range in a strong sense.

\end{abstract}

\section*{Introduction}

The aim of the present paper is to exhibit a space with the
properties described in the abstract. The norming set $W$ of the
space $\X$ is saturated with constraints and it is very similar to
the corresponding one in \cite{AM}. As it is pointed out in
\cite{AM} the method of saturation under constraints is suitable
for defining spaces with hereditary heterogeneous structure
(\cite{OS1}, \cite{OS2}). The basic ingredients of the norming set
$W$ are the following. First the unconditional frame is the ball
of the dual $T^*$ of Tsirelson space \cite{FJ},\cite{T}; namely $W$ is a
subset of $B_{T^*}$ which satisfies the following properties. As
in \cite{AM} it is closed in the operations
$(\frac{1}{2^n},\mathcal{S}_n,\al),
(\frac{1}{2^n},\mathcal{S}_n,\be)$ which create the type I$_\al$,
type I$_\be$ functionals respectively. Furthermore it includes two
types of special functionals denoted as type II$_+$ and type
II$_-$ functionals. The type II$_-$ functionals are designed to
impose the rich spreading model structure in the space $\X$, while
the type II$_+$ functionals serve a double purpose. First they are
a tool for finding $c_0$ spreading models in every subspace of
$\X$. The $c_0$ spreading models are the fundamental initial
ingredient for the ultimate construction. The second role of the
type II$_+$ functionals is to show that the space $\X$ is not
tight by range. We recall that recently V. Ferenczi and Th.
Schlumprecht have presented in \cite{FS} a variant of
Gowers-Maurey HI space (\cite{GM}) which is HI and not tight by
range.

Since the norming set $W$ is similar to the one in \cite{AM} many
of the critical norm evaluations in the present paper are
identical with the corresponding ones in \cite{AM}. The main
difference of the present construction from the one in \cite{AM}
concerns the ``combinatorial result'' which is a Ramsey type
result yielding $c_0$ spreading models. For the proof of this
result, type II$_+$ functionals are a key ingredient.

We pass to a more detailed description of the properties of the
space $\X$.

\begin{thm*}

The space $\X$ is reflexive, HI and hereditarily unconditional
spreading model universal.

\end{thm*}

The latter means that there exists a universal constant $C>0$ such
that the following holds. For every subspace $Y$ of $\X$ there
exists  a seminormalized weakly null sequence $\{x_n\}_n$
admitting spreading models $C$-equivalent to all spreading
suppression unconditional  sequences. The fundamental property of $\{x_n\}_n$
deriving its spreading model universality is that for every
Schreier set $F\subset\mathbb{N}$ the finite sequence
$\{x_n\}_{n\in F} \stackrel{C}{\sim} \{u_n\}_{n\in F}$, where
$\{u_n\}_n$ denotes Pe\l czynski's universal unconditional basis
\cite{P},\cite{LT}.

The second property of $\X$ is that it is sequentially minimal. We
recall, from \cite{FR}, that a Banach space $X$ with a basis is
sequentially minimal, if in every  infinite dimensional block
subspace $Y$ of $X$ there exists a block sequence
$\{x_n^{(Y)}\}_n$ satisfying the following. In every subspace $Z$
of $X$ there exists a Schauder basic sequence $\{z_k\}_k$
equivalent to a subsequence of $\{x_{n}^{(Y)}\}_n$. Also recall that a Banach space $X$ with a basis is called tight by range, if whenever $\{y_k\}_k$, $\{z_k\}_k$ are block sequences in $X$ with $\ran y_k\cap \ran z_m = \varnothing$ for all $k,m$, then if $Y = [\{y_k\}_k]$ and $Z = [\{z_k\}_k]$, none of these two spaces embeds into the other. A dichotomy of
V. Ferenczi - Ch. Rosendal classification program \cite{FR} yields
that every Banach space $X$ with a Schauder basis $\{e_n\}_n$
either contains a block subspace which is tight by range or a
sequentially minimal subspace. As consequence of this dichotomy,
$\X$ is not tight by range, however we also prove this by showing that the following stronger fact holds.

\begin{thm*}

Every $Y$ block subspace of $\X$ contains a seminormalized block
sequence $\{x_n\}_n$ satisfying the following. There exists an
isomorphism $T:\X\rightarrow\X$ (necessarily onto) such that
$T(x_{2n-1}) = x_{2n}$ for $n\inn$.

\end{thm*}

The above result is a direct consequence of the structure imposed
to the norming set $W$ and hence to the space $\X$, in order to
achieve the rich spreading model structure. In particular the
following is proved.

\begin{prp*}

Let $Y$ be a block subspace of $\X$. Then there exist
$\{x_n,y_n\}_n$ in Y, $\{f_n,g_n\}_n$ such that $f_n,g_n$ belong to
$W$, $\ran x_n = \ran f_n$, $\ran y_n = \ran g_n$, $x_n < y_n <
x_{n+1}$, $\{x_n\}_n, \{y_n\}_n$ are seminormalized, $f_n(x_n) =
1, g_n(y_n) = 1$ and $\{f_n + g_n\}_n$ generates a $c_0$ spreading
model while $\{x_n - y_n\}_n$ does not generate an $\ell_1$
spreading model.

\end{prp*}

The above proposition yields that there exists a strictly singular
operator  $S:\X\rightarrow\X$ with $S(x_n) = x_n - y_n$ and
$S(y_n) = x_n - y_n$ (see \cite{ADT}). As is explained in
\cite{FR}, the sequences $\{x_n\}_n$, $\{y_n\}_n$ are equivalent.
It is also easy to see that $I-S$ is an isomorphism, satisfying
the conclusion of the above theorem.

Every operator in the space $\X$ is of the form $T = \la I + S$
with $S$ strictly singular. We recall that one of the main
properties of the space in \cite{AM}, is that the composition of
any three strictly singular operators is a compact one. It is shown
that the space $\X$ fails such a property, by proving that in any
block subspace there exists a strictly singular operator, which is
not polynomially compact. The proof of this result is directly
linked to the variety of spreading models appearing in every block
subspace of $\X$.

The paper is organized as follows. In the first section basic notions used throughout the paper are introduced. The second section is devoted to
the definition of the norming set $W$ of the space $\X$, a brief
discussion is also included concerning the role of its
ingredients. The third section concerns some basic norm
evaluations on special convex combinations, which are identical to
the corresponding estimates from \cite{AM}. The fourth section
introduces the definition of the $\al, \be$ indices, which are
defined in the same manner as in \cite{AM} and related results. In
the fifth section, a combinatorial result is stated and proven
and it is used in the sixth section to establish the existence of
$c_0$ spreading models. In the seventh section the structure of the
spreading models of the space $\X$ is studied. In the eighth and
final section it is proven that the space is sequentially minimal,
it is not tight by range and it admits strictly singular non
polynomially compact operators.

\section{Preliminaries}
\label{section1}

\subsection*{Spreading models} The notion of a spreading model was invented by A. Brunel and L. Sucheston in \cite{BS} and has become a central concept in Banach space theory. Below we include the definition and some basic facts concerning spreading models.

\begin{dfn}Let $(X,\|\cdot\|)$ be a Banach space and $(E,\|\cdot\|_*)$ a seminormed
space. Let $\{x_n\}_{n}$ be a bounded sequence in $X$ and
$\{e_n\}_{n}$ a sequence in $E$. We say that {\em $\{x_n\}_{n}$
generates $\{e_n\}_{n}$ as a spreading model}, if there exists
a sequence of positive reals $\{\de_n\}_{n}$ with
$\de_n\searrow 0$, such that for every $n\inn,\; n\leqslant k_1
<\ldots<k_n$ and every choice $\{a_i\}_{i=1}^n\subset[-1,1]$ the
following holds:

\begin{equation*}
\Big|\big\|\sum_{i=1}^na_ix_{k_i}\big\| -
\big\|\sum_{i=1}^na_ie_i\big\|_*\Big| < \de_n.
\end{equation*}

\end{dfn}

In the sequel, by saying that a sequence $\{x_k\}_k$ generates an $\ell_p, 1\leqslant p <\infty$ (resp. $c_0$) spreading model, we shall mean that it generates a spreading model that is equivalent to the usual basis of $\ell_p$ (resp. $c_0$).

Brunel and Sucheston proved that every bounded sequence in a
Banach space, has a subsequence which generates a spreading model.
The main property of spreading models is that they are spreading
sequences, i.e. for every $n\inn,\; k_1 <\ldots<k_n$ and every
choice $\{a_i\}_{i=1}^n\subset\mathbb{R}$ we have
$\|\sum_{i=1}^na_ie_i\|_* = \|\sum_{i=1}^na_ie_{k_i}\|_*$.

Spreading sequences are classified into four categories, with respect
to their norm properties. These are the trivial, the
unconditional, the singular and the non unconditional Schauder
basic spreading sequences (see \cite{AKT}).

A spreading sequence $\{e_n\}_{n}$ is called trivial, if the
seminorm on the space generated by the sequence is not actually a
norm. In this case, Proposition 13 from \cite{AKT} yields the following. If $E$ is the vector space generated by
$\{e_n\}_{n}$ and $\mathcal{N} = \{x\in E: \|x\|_* = 0\}$,
then {\small $E$}$/_\mathcal{N} $ has dimension at most 1. It is also
worth mentioning that a sequence in a Banach space $X$ generates a
trivial spreading model, if and only if it has a norm convergent
subsequence. For more details see \cite{AKT}, \cite{BL}. From now on, we will
only refer to non trivial spreading models.

A spreading sequence is called singular if it is not trivial and not
Schauder basic. The definition of the other two cases is the
obvious one.

\subsection*{The Schreier families} The Schreier families is
an increasing sequence of compact families of finite subsets of the
natural numbers, which first appeared in \cite{AA}, and it is inductively defined in the
following manner.

Set $\mathcal{S}_0 = \big\{\{n\}: n\inn\big\}$ and $\mathcal{S}_1
= \{F\subset\mathbb{N}: \#F\leqslant\min F\}$.

Suppose that $\Sn$ has been defined and set $\mathcal{S}_{n+1} =
\{F\subset\mathbb{N}: F = \cup_{j = 1}^k F_j$, where $F_1 <\cdots<
F_k\in\Sn$ and $k\leqslant\min F_1\}$

If for $n,m\inn$ we set $\Sn*\mathcal{S}_m = \{F\subset\mathbb{N}:
F = \cup_{j = 1}^k F_j$, where $F_1 <\cdots< F_k\in\mathcal{S}_m$
and $\{\min F_j: j=1,\ldots,k\}\in\Sn\}$, then it is well known \cite{AD} and follows easily by induction
that $\Sn*\mathcal{S}_m = \mathcal{S}_{n+m}$.

A sequence of vectors $x_1<\cdots<x_k$ in $c_{00}$ is said to be $\Sn$-admissible if
$\{\min\supp x_i: i=1,\ldots,k\}\in\Sn$.

\subsection*{The suppression unconditional universal basis of Pe\l czy\' nski}

Let $\{x_k\}_k$ be a norm dense sequence in the unit sphere of
$C[0,1]$. Denote by $\{u_k\}_k$ the unit vector basis of $c_{00}$
and define $\|\cdot\|_u$ on $c_{00}$ as follows.
\begin{equation*}
\|\sum_{k=1}^n\al_ku_k\|_u = \sup\big\{\|\sum_{k\in F}\al_kx_k\|:
F\subset\{1,\ldots,n\}\big\}
\end{equation*}

Let $U$ be the completion of $(c_{00},\|\cdot\|_u)$. Then
$\{u_k\}_k$ is a suppression unconditional Schauder basis for $U$,
such that for any $\{y_k\}_k$ suppression unconditional Schauder
basic sequence and $\e>0$, there exists a subsequence of
$\{u_k\}_k$, which is $(1+\e)$-equivalent to $\{y_k\}_k$.

The sequence $\{u_k\}_k$ is called the unconditional basis of Pe\l
czy\' nski (see \cite{P}).

\section{The norming set of the space $\X$.}

In this section we define the norming set $W$ of the space $\X$.
As in \cite{AM}, this set is defined with the use of the sequence
$\{\Sn\}_n$ and also families of
$\Sn$-admissible functionals and the set $W$ will be a subset of
the norming set $W_T$ of Tsirelson space. The key difference
between the construction in \cite{AM} and the present one, is the
way functionals of type II are defined, which yields the
properties of the space $\X$.

\begin{ntt} 

Let $G\subset c_{00}$. A vector $f\in G$ is said to be an average
of size $s(f) = n$, if there exist $f_1,\ldots,f_d\in G,
d\leqslant n$, such that $f = \frac{1}{n}(f_1+\cdots+f_d)$.

A sequence $\{f_j\}_j$ of averages in $G$ is said to be very fast
growing, if $f_1<f_2<\ldots$, $s(f_j)>2^{\max\supp f_{j-1}}$ and
$s(f_j) > s(f_{j-1})$ for $j>1$.
\end{ntt}

\subsection*{The coding function} Choose $L_0 = \{\ell_k:
k\inn\}, \ell_1 > 9$ an infinite subset of the natural numbers such that:
\begin{enumerate}

\item[(i)] For any $k\inn$ we have that $\ell_{k+1} > 2^{2\ell_k}$
and

\item[(ii)] $\sum_{k=1}^\infty
\frac{1}{2^{\ell_k}}<\frac{1}{1000}$.

\end{enumerate}
Decompose $L_0$ into further infinite subsets $L_1, L_2, L_3$. Set
\begin{eqnarray*}
\mathcal{Q} &=&
\big\{\big(f_1,\ldots,f_m\big): m\inn, f_1 <\ldots <f_m\in c_{00}\\
&&\text{with}\;f_k(i)\in\mathbb{Q},\;\text{for}\;i\inn,
k=1,\ldots,m\}
\end{eqnarray*}
Choose a one to one function $\sigma:\mathcal{Q}\rightarrow L_2$,
called the coding function, such that for any
$\big(f_1,\ldots,f_m\big)\in\mathcal{Q}$, we have that
\begin{equation*}
\sigma\big(f_1,\ldots,f_m\big) >
2^{\frac{1}{\|f_m\|_0}}\cdot\max\supp f_m
\end{equation*}

\begin{rmk}
If we set $L = L_1\cup L_2$, For any $n\inn$ we have that
$\#L\cap\{n,\ldots,2^{2n}\}\leqslant 1$, moreover for every $n\in L_3$, we have that $L\cap\{n,\ldots,2^{2n}\}=\varnothing$. \label{lacunarity}
\end{rmk}

\subsection*{The norming set} The norming set $W$ is defined to
be the smallest subset of $c_{00}$ satisfying the following
properties:\vskip3pt

\noindent {\bf 1.} The set $\{\substack{+\\[-2pt]-} e_n\}_{n\inn}$
is a subset of $W$, for any $f\in W$ we have that $-f\in W$, for
any $f\in W$ and any $E$ interval of the natural numbers we have that
$Ef\in W$ and $W$ is closed under rational convex
combinations. Any $f = \substack{+\\[-2pt]-} e_n$ will be called a functional of type 0.\vskip3pt

\noindent {\bf 2.} The set $W$ is closed in the
$(\frac{1}{2^n},\mathcal{S}_n,\al)$ operation, i.e. it contains
any functional $f$ which is of the form $f =
\frac{1}{2^n}\sum_{q=1}^d\al_q$, where $\{\al_q\}_{q=1}^d$ is an
$\Sn$-admissible and very fast growing sequence of $\al$-averages
in $W$. If $E$ is an interval of the
natural numbers, then $g = \substack{+\\[-2pt]-}Ef$ is called a functional of type I$_\al$, of
weight $w(g) = n$.\vskip3pt

\noindent {\bf 3.} The set $W$ is closed in the
$(\frac{1}{2^n},\mathcal{S}_n,\be)$ operation, i.e. it contains
any functional $f$ which is of the form $f =
\frac{1}{2^n}\sum_{q=1}^d\be_q$, $\{\be_q\}_{q=1}^d$ is an
$\Sn$-admissible and very fast growing sequence of $\be$-averages
in $W$. If $E$ is an interval of the
natural numbers, then $g = \substack{+\\[-2pt]-}Ef$ is called a functional of type I$_\be$, of
weight $w(g) = n$.\vskip3pt

\noindent {\bf 4.} For any special sequence $\{f_q,g_q\}_{q=1}^d$
in $W$ and $F\subset\{1,\ldots,d\}$ such that $2(\#F)\leqslant
\min\supp f_{\min F}$, the set $W$ contains any functional $f$
which is of the form $f = \frac{1}{2}\sum_{q\in F}(f_q + g_q)$.

If $E$ is an interval of the natural numbers, then $g = \substack{+\\[-2pt]-}Ef$
is called a functional of type II$_+$ with weights $\widehat{w}(g)
= \{w(f_q) : q\in F, \ran(f_q + g_q)\cap
E\neq\varnothing\}$.\vskip3pt

\noindent {\bf 5.} For any special sequence $\{f_q,g_q\}_{q=1}^d$
in $W$ and $F\subset\{1,\ldots,d\}$ such that $2(\#F)\leqslant
\min\supp f_{\min F}$ and $\{\la_q\}_{q\in F}\subset\mathbb{Q}$
with $\|\sum_{q\in F}\la_qu_q^*\|_u\leqslant 1$, where
$\{u_k^*\}_k$ denotes the biorthogonals of the unconditional basis
of Pe\l czy\' nski, the set $W$ contains any functional $f$ which
is of the form $f = \frac{1}{2}\sum_{q\in F}\la_q(f_q-g_q)$.

If $E$ is an interval of the natural numbers, then $g = \substack{+\\[-2pt]-}Ef$
is called a functional of type II$_-$ with weights $\widehat{w}(g)
= \{w(f_q) : q\in F, \ran(f_q - g_q)\cap
E\neq\varnothing\}$.\vskip3pt

We call a functional $f\in W$ which is either of type II$_+$ or of
type II$_-$, a functional of type II.\vskip3pt

For $d\inn$, a sequence of pairs of functionals of type I$_\al$
$\{f_q,g_q\}_{q=1}^d$, is called a special sequence if
\begin{equation}
f_1<g_1<f_2<g_2<\cdots<f_d<g_d
\end{equation}
\begin{equation}
w(f_q) = w(g_q)\quad\text{for}\quad q=1,\ldots,d
\end{equation}
\begin{equation}
w(f_1)\in
L_1\quad\text{and}\quad\sigma(f_1,g_1,f_2,g_2\ldots,f_{q-1},g_{q-1})
= w(f_q)\;\text{for}\;1<q\leqslant d
\end{equation}\vskip3pt

We call an $\al$-average any average $\al\in W$ of the form $\al =
\frac{1}{n}\sum_{j=1}^df_j, d\leqslant n$, where
$f_1<\cdots<f_d\in W$.\vskip3pt

We call a $\be$-average any average $\be\in W$ of the form $\be =
\frac{1}{n}\sum_{j=1}^df_j, d\leqslant n$, where
$f_1,\ldots,f_d\in W$ are functionals of type II, with pairwise
disjoint weights $\widehat{w}(f_j)$.\vskip3pt

In general, we call a convex combination any $f\in W$ that is not
of type 0, I$_\al$, I$_\be$ or II.\vskip3pt

A sequence of pairs of functionals of type I$_\al$ $b =
\{f_q,g_q\}_{q=1}^\infty$ is called a special branch, if
$\{f_q,g_q\}_{q=1}^d$ is a special sequence for all $d\inn$. We
denote the set of all special branches by $\mathcal{B}$.

If $b = \{f_q,g_q\}_{q=1}^\infty\in\mathcal{B}$, we denote by $b_+
= \{f_q+g_q: q\inn\}$ and $b_- = \{f_q-g_q: q\inn\}$.

For $x\in c_{00}$ define $\|x\| = \sup\{f(x): f\in W\}$ and $\X =
\overline{(c_{00}(\mathbb{N}),\|\cdot\|)}$. Evidently $\X$ has a
bimonotone basis.

\begin{rmk} The norming set $W$ can be inductively constructed to
be the union of an increasing sequence of subsets
$\{W_m\}_{m=0}^\infty$ of $c_{00}$, where $W_0 = \{\substack{+\\[-2pt]-}
e_n\}_{n\inn}$ and if $W_m$ has been constructed, then set
$W_{m+1}^\al$ to be the closure of $W_m$ under taking
$\al$-averages, $W_{m+1}^{\text{I}_\al}$ to be the closure of
$W_{m+1}^\al$ under taking type I$_\al$ functionals,
$W_{m+1}^{\text{I}_\be}$ to be the closure of
$W_{m+1}^{\text{I}_\al}$ under taking type I$_\be$ functionals,
$W_{m+1}^{\text{II}}$ to be the closure of
$W_{m+1}^{\text{I}_\be}$ under taking type II$_+$ and type II$_-$ functionals,
$W_{m+1}^\be$ to be the closure of $W_{m+1}^{\text{II}}$ under
taking $\be$-averages and finally $W_{m+1}$ to be the closure of
$W_{m+1}^\be$ under taking rational convex combinations.
\label{remark1.2}
\end{rmk}

\subsection*{The features of the space $\X$}

Before proceeding to the study of the space $\X$, it is probably
useful to explain the role of the specific  ingredients in the
definition of the norming set $W$. First, as we have mentioned in
the introduction, we will use saturation under constraints in a
similar manner as in \cite{AM}. This yields the type I$_\al$,
I$_\be$ functionals and the indices $\adx, \bdx$ for block
sequences $\{x_k\}_k$ in $\X$, which are defined as in \cite{AM}.
As the familiar reader would observe, the special functionals in
the present construction differ from the corresponding ones in
\cite{AM}. This is due to the desirable main property of the space
$\X$, namely that every subspace contains a sequence admitting all
unconditional spreading sequences as a spreading model. This is
related to property 5 in the above definition of the norming set
$W$.

What requires further discussion are the type II$_+$ functionals.
The primitive role of them is to allow to locate in every block
subspace a seminormalized block sequence generating a $c_0$
spreading model. This follows from the next proposition.

\begin{prp*} Let $\{x_k\}_k$ be a seminormalized block sequence in
$\X$ such that the following hold.
\begin{itemize}

\item[(i)] $\adx = 0$ and $\bdx = 0$

\item[(ii)] For every special branch $b =
\{f_q,g_q\}_{q=1}^\infty$
\begin{equation*}
\lim_k\sup\big\{|f_q(x_k)|\vee|g_q(x_k)|: q\inn\big\} = 0
\end{equation*}

\end{itemize}
Then there exists a subsequence $\{x_{k_n}\}_n$ of $\{x_k\}_k$
generating a $c_0$ spreading model.

\end{prp*}
Note that in \cite{AM}, property (i) is sufficient for a sequence
to have a subsequence generating a $c_0$ spreading model. However,
in the present paper this is not the case and the special
functionals of type II$_+$ are crucial for establishing the existence, in every block subspace, of block sequences satisfying (i) and
(ii) in the above proposition.

As consequence, we obtain that in every block subspace there
exists a block sequence generating a $c_0$ spreading model.

In figure \ref{figure1} we describe how the type II$_+$ and type II$_-$ functionals are constructed, starting with a special branch $b = \{f_q,g_q\}_q$.

\begin{figure}[htb]
\centering
\includegraphics[scale=0.65]{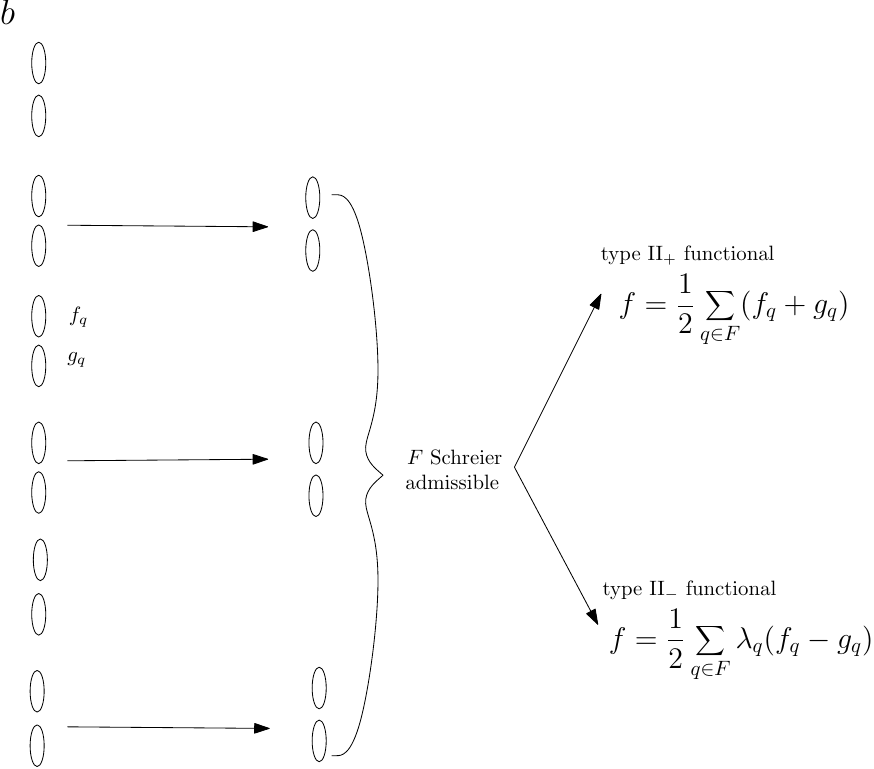}
\caption{}
\label{figure1}

\end{figure}

As in
\cite{AM}, from the $c_0$ spreading model one can pass to exact
nodes (see Def. \ref{definitionexactnode}) $\{x_k,y_k,f_k,g_k\}$,
with $\{f_k,g_k\}_{k=1}^\infty$ defining a special branch. The sequences $\{x_k\}_k, \{y_k\}_k$ and $\{x_k+y_k\}_k$ all admit only $\ell_1$ as a spreading model. The above property is imposed by the type II$_+$ functionals.
On the other hand, the type II$_-$ functionals yield that the sequence $\{x_k - y_k\}_k$ is spreading model universal, namely for every spreading and unconditional sequence $\{d_k\}_k$, there exists a subsequence of $\{x_k - y_k\}_k$, admitting a spreading model equivalent to $\{d_k\}_k$. A secondary
role of the type II$_+$ special functionals is to determine
intertwined equivalent sequences $\{v_k,w_k\}_k$. Those are
subsequences of the above described sequence $\{x_k,y_k\}_k$.

As in \cite{AM}, the norming set of the space $\X$ is a subset of
the unit ball of the dual $T^*$ of Tsirelson space (see
\cite{FJ}). Moreover, most of the critical norm evaluations are
identical with those in \cite{AM}.

\section{basic evaluations for special convex combinations}

In this section we present some results concerning estimations of
the norm of special convex combinations. These estimations are
crucial throughout the rest of the paper, as like in \cite{AM},
special convex combinations are one of the main tools used to
establish the properties of the space $\X$.

\begin{dfn} Let $x = \sum_{k\in F}c_ke_k$ be a vector in $c_{00}$.
Then $x$ is said to be a $(n,\e)$ basic special convex combination
(or a $(n,\e)$ basic s.c.c.) if:

\begin{enumerate}

\item[(i)] $F\in\Sn, c_k\geqslant 0$, for $k\in F$ and $\sum_{k\in
F}c_k = 1$.

\item[(ii)] For any $G\subset F, G\in\mathcal{S}_{n-1}$, we have
that $\sum_{k\in G}c_k < \e$.

\end{enumerate}

\end{dfn}

The proof of the next proposition can be found in \cite{AT},
Chapter 2, Proposition 2.3.

\begin{prp}
For any  infinite subset of the natural numbers $M$, any $n\inn$ and
$\e>0$, there exists $F\subset M, \{c_k\}_{k\in F}$, such that $x
= \sum_{k\in F}c_ke_k$ is a $(n,\e)$ basic
s.c.c.\label{sccexistence}
\end{prp}

\begin{dfn}
Let $x_1 <\cdots<x_m$ be vectors in $c_{00}$ and $\psi(k) =
\min\supp x_k$, for $k=1,\ldots,m$. Then $x = \sum_{k=1}^mc_kx_k$
is said to be a $(n,\e)$ special convex combination (or $(n,\e)$
s.c.c.), if $\sum_{k=1}^mc_ke_{\psi(k)}$ is a $(n,\e)$ basic
s.c.c.\label{defscc}
\end{dfn}

The proof of the following result can be found in \cite{AM}, Proposition 2.5.

\begin{prp} Let $x = \sum_{k\in F}c_ke_k$ be a $(n,\e)$ basic
s.c.c. and $G\subset F$. Then the following holds.

\begin{equation*}
\|\sum_{k\in G}c_ke_k\|_T \leqslant\frac{1}{2^n}\sum_{k\in G}c_k +
\e
\end{equation*}\label{Tsirelsonscc}
\end{prp}

The next proposition is identical to Corollary 2.8 from \cite{AM}.

\begin{prp} Let $\{x_k\}_{k}$ be a block sequence in $\X$ such that
$\|x_k\|\leqslant 1, \{c_k\}_k\subset\mathbb{R}$ and $\phi(k) =
\max\supp x_k$ for all $k$. Then:
\begin{equation}
\|\sum_kc_kx_k\| \leqslant 6\|\sum_kc_ke_{\phi(k)}\|_T
\end{equation}\label{tsirelsondominated}
\end{prp}

The following corollary is an easy consequence of Propositions
\ref{Tsirelsonscc} and \ref{tsirelsondominated} and its proof can
be found in \cite{AM}, Corollary 2.9.

\begin{cor} Let $x = \sum_{k=1}^mc_kx_k$ be a $(n,\e)$ s.c.c. in $\X$, such
that $\|x_k\|\leqslant 1$, for $k=1,\ldots,m$. If
$F\subset\{1,\ldots,m\}$, then
\begin{equation*}
\|\sum_{k\in F}c_kx_k\| \leqslant \frac{6}{2^n}\sum_{k\in F}c_k +
12\e.
\end{equation*}
In particular, we have that $\|x\|\leqslant \frac{6}{2^n} +
12\e$.\label{sccnormestimate}
\end{cor}

The proof of the next corollary is based on Corollary \ref{sccnormestimate}. It's proof is identical to the one of Corollary 2.10 from \cite{AM}.

\begin{cor} The basis of $\X$ is shrinking.\label{shrinking}
\end{cor}

The definition of the norming set yields the following result, the proof of which can be found in \cite{AM}, Corollary 2.11.

\begin{prp} The basis of $\X$ is boundedly complete.
\end{prp}

Combining the previous two results with R. C. James' well known result \cite{J}, we conclude the following.

\begin{cor} The space $\X$ is reflexive.
\end{cor}

Rapidly increasing sequences are defined in the exact same manner,
as in \cite{AM}, Definition 2.13.

\begin{dfn} Let $C\geqslant 1, \{n_k\}_k$ be strictly
increasing natural numbers. A block sequence $\{x_k\}_k$ is called a
$(C,\{n_k\}_k)$ $\al$-rapidly increasing sequence (or
$(C,\{n_k\}_k)\;\al$-RIS) if the following hold.
\begin{enumerate}

\item[(i)] $\|x_k\| \leqslant C,\quad
\frac{1}{2^{n_{k+1}}}\max\supp x_k < \frac{1}{2^{n_k}}\quad$ for
all $k$.

\item[(ii)] For any functional $f$ in $W$ of type I$_\al$ of
weight $w(f) = n$, for any $k$ such that $n<n_k$, we have that
$|f(x_k)| < \frac{C}{2^n}$.

\end{enumerate}
\end{dfn}

\begin{dfn}
Let $n\inn, C\geqslant 1, \theta>0$. A vector $x\in\X$ is called a
$(C,\theta,n)$ vector if the following hold. There exist
$0<\e<\frac{1}{36C2^{3n}}$, and $\{x_k\}_{k=1}^m$ with $\|x_k\| \leqslant C$ for $k=1,\ldots,m$ such that

\begin{enumerate}

\item[(i)] $\min\supp x_1 \geqslant 8C2^{2n}$

\item[(ii)] There exist $\{c_k\}_{k=1}^m\subset[0,1]$ such that
$\sum_{k=1}^mc_kx_k$ is a $(n,\e)$ s.c.c.

\item[(iii)] $x = 2^n\sum_{k=1}^mc_kx_k$ and $\|x\| \geqslant
\theta$

\end{enumerate}
If moreover there exist $\{n_k\}_{k=1}^m$ strictly increasing
natural numbers with $n_1>2^{2n}$ such that $\{x_k\}_{k=1}^m$ is
$(C,\{n_k\}_{k=1}^m)\;\al$-RIS, then $x$ is called a
$(C,\theta,n)$ exact vector.

\end{dfn}

\begin{rmk}
Let $x$ be a $(C,\theta,n)$ vector in $\X$. Then, using Corollary \ref{sccnormestimate} we conclude that $\|x\| <
7C$.\label{exactvectornorm}
\end{rmk}

\section{The $\al,\be$ indices}
The $\al$ and $\be$ indices concerning block sequences in $\X$,
are identically defined, as in \cite{AM}, Definitions 3.1 and 3.2.
Note that in \cite{AM}, the $\al, \be$ indices are sufficient to
fully describe the spreading models admitted by block sequences.
In the present paper, this is not the case. However, the $\al,
\be$ indices retain an important role in determining what
spreading models a block sequence generates.

\begin{dfn}
Let $\{x_k\}_k$ be a block sequence in $\X$ that satisfies the
following. For any $n\inn$, for any very fast growing sequence
$\{\al_q\}_q$ of $\al$-averages in $W$ and for any 
increasing sequence of subsets of the natural numbers $\{F_k\}_k$, such that
$\{\al_q\}_{q\in F_k}$ is $\Sn$-admissible, the following holds.
For any $\{x_{n_k}\}_k$ subsequence of $\{x_k\}_k$, we have that
$\lim_k\sum_{q\in F_k}|\al_q(x_{n_k})| = 0$.

Then we say that the $\al$-index of $\{x_k\}_k$ is zero and write
$\adx = 0$. Otherwise we write $\adx > 0$.
\end{dfn}

\begin{dfn}
Let $\{x_k\}_k$ be a block sequence in $\X$ that satisfies the
following. For any $n\inn$, for any very fast growing sequence
$\{\be_q\}_q$ of $\be$-averages in $W$ and for any 
increasing sequence of subsets of the natural numbers $\{F_k\}_k$, such that
$\{\be_q\}_{q\in F_k}$ is $\Sn$-admissible, the following holds.
For any $\{x_{n_k}\}_k$ subsequence of $\{x_k\}_k$, we have that
$\lim_k\sum_{q\in F_k}|\be_q(x_{n_k})| = 0$.

Then we say that the $\be$-index of $\{x_k\}_k$ is zero and write
$\bdx = 0$. Otherwise we write $\bdx > 0$.
\end{dfn}

\begin{rmk}
Let $\{x_k\}_k$ be a block sequence in $\X$ and $\{E_k\}_k$ be an increasing sequence
of intervals of the natural numbers with $E_k\subset \ran x_k$ for all $k\inn$. Set $y_k = E_kx_k$.

\begin{enumerate}

\item[(i)]
If $\adx = 0$, then $\ady = 0$.

\item[(ii)]
If $\bdx = 0$, then $\bdy = 0$.

\end{enumerate}
\label{indexinterval}
\end{rmk}

\begin{rmk}
Let $\{x_k\}_k$, $\{y_k\}_k$ be block sequence such that if $z_k = x_k + y_k$, $\{z_k\}_k$ is also a block sequence.

\begin{enumerate}

\item[(i)]
If $\adx = 0$ and $\ady = 0$, then $\adz = 0$.

\item[(ii)]
If $\bdx = 0$ and $\bdy = 0$, then $\bdz = 0$.

\end{enumerate}
\label{indexsum}
\end{rmk}

\begin{rmk} Let $\{x_k\}_k$ be a block sequence in $\X$ and $\{F_k\}_k$ be an increasing sequence
of subsets of the natural numbers and $\{c_i\}_{i\in F_k}\subset[0,1]$ with $\sum_{i\in F_k}c_i = 1$
for all $k\inn$. Set $y_k = \sum_{i\in F_k}c_ix_i$.

\begin{enumerate}

\item[(i)]
If $\adx = 0$, then $\ady = 0$.

\item[(ii)]
If $\bdx = 0$, then $\bdy = 0$.

\end{enumerate}
\label{indexconvex}
\end{rmk}

The following two Propositions are proven in \cite{AM}, Proposition 3.3.

\begin{prp} Let $\{x_k\}_k$ be a block sequence in $\X$. Then the
following assertions are equivalent.
\begin{enumerate}

\item[(i)] $\adx = 0$

\item[(ii)] For any $\e>0$ there exists $j_0\inn$ such that for
any $j\geqslant j_0$ there exists $k_j\inn$ such that for any
$k\geqslant k_j$, and for any $\{\al_q\}_{q=1}^d$ $\Sj$-admissible
and very fast growing sequence of $\al$-averages such that
$s(\al_q) > j_0$, for $q=1,\ldots,d$, we have that
$\sum_{q=1}^d|\al_q(x_k)| < \e$.

\end{enumerate}
\label{aindexequivalent}
\end{prp}

\begin{prp} Let $\{x_k\}_k$ be a block sequence in $\X$. Then the
following assertions are equivalent.
\begin{enumerate}

\item[(i)] $\bdx = 0$

\item[(ii)] For any $\e>0$ there exists $j_0\inn$ such that for
any $j\geqslant j_0$ there exists $k_j\inn$ such that for any
$k\geqslant k_j$, and for any $\{\be_q\}_{q=1}^d$ $\Sj$-admissible
and very fast growing sequence of $\be$-averages such that
$s(\be_q) > j_0$, for $q=1,\ldots,d$, we have that
$\sum_{q=1}^d|\be_q(x_k)| < \e$.

\end{enumerate}
\label{bindexequivalent}
\end{prp}

The next Proposition is similar to Proposition 3.5 from \cite{AM}.

\begin{prp}
Let $\{x_k\}_k$ be a seminormalized block sequence in $\X$, such
that either $\adx > 0$, or $\bdx > 0$. Then there exists a
subsequence $\{x_{n_k}\}_k$ of $\{x_k\}_{k\inn}$, that generates
an $\ell_1^n$ spreading model, for every $n\inn$.

In particular, there exists $\theta>0$ such that for any
$k_0,n\inn$, there exists  a $(C,\theta,n)$ vector $x$ supported by
$\{x_k\}_k$ with $\min\supp x\geqslant k_0$, where $C = \sup_k\|x_k\|$.

If moreover $\{x_k\}_k$ is $(C^\prime,\{n_k\})\;\al$-RIS, then for every
$n,k_0\inn$ there exists  a $(C^{\prime\prime},\theta,n)$ exact vector $x$
supported by $\{x_k\}_k$ with $\min\supp x\geqslant
k_0$, where $C^{\prime\prime} = \max\{C, C^\prime\}$.\label{ell1nandexactvectorexist}
\end{prp}

The proof of the following lemma, is identical to Lemma 3.7 from \cite{AM}.

\begin{lem} Let $x = 2^n\sum_{k=1}^mc_kx_k$ be a $(C,\theta,n)$ vector in $\X$. Let also $\al$ be
an $\al$-average and set $G_\al = \{k: \ran\al\cap\ran
x_k\neq\varnothing\}$. Then the following holds.
\begin{equation}
|\al(x)| < \min\big\{\frac{C2^n}{s(\al)}\sum_{k\in G_\al}c_k,\;
\frac{6C}{s(\al)}\sum_{k\in G_\al}c_k +
\frac{1}{3\cdot2^{2n}}\big\} + 2C2^n\max\{c_k: k\in G_\al\}
\end{equation}
\end{lem}

The next lemma is proven in \cite{AM}, Lemma 3.8.

\begin{lem}
Let $x$ be a $(C,\theta,n)$ vector in $\X$. Let also
$\{\al_q\}_{q=1}^d$ be a very fast growing and $\Sj$-admissible
sequence of $\al$-averages with $j<n$. Then the following holds.

\begin{equation}
\sum_{q=1}^d|\al_q(x)| < \frac{6C}{s(\al_1)} + \frac{1}{2^n}
\end{equation}\label{exactvectoraestimate}
\end{lem}

The following corollary is an immediate consequence of Lemma \ref{exactvectoraestimate} and it is  similar with Proposition 3.10 from \cite{AM}.

\begin{cor} Let $x$ be a $(C,\theta,n)$ vector in $\X$. Let also $f$ be a functional of type I$_\al$ in $W$
with $w(f) = j <n$. Then the following holds
\begin{equation}
|f(x)| < \frac{6C + 1/2^n}{2^j}
\end{equation}\label{vectorris}
\end{cor}

Combining Lemma \ref{exactvectoraestimate} with Corollary \ref{vectorris} we conclude the following.

\begin{cor}
Let $\{x_k\}_k$ be a block sequence in $\X$, such that $x_k$ is a
$(C,\theta,n_k)$ vector and $\{n_k\}_k$ is strictly
increasing. Then $\adx = 0$. Moreover, passing if necessary to a subsequence, $\{x_k\}_k$ is $(7C, \{n_k\}_k)\;\al$-RIS.\label{corexactvetorazero}
\end{cor}

\begin{ntt}
Let $x = 2^n\sum_{k=1}^mc_kx_k$ be a $(C,\theta,n)$ exact vector in $\X$, where $\{x_k\}_{k=1}^m$ is $(C,\{n_k\}_{k=1}^m)\;\al$-RIS.
Let also $f = \frac{1}{2}\sum_{q\in F}(f_q + g_q)$ be a type II$_+$
functional (or $f = \frac{1}{2}\sum_{q\in F}\la_q(f_q - g_q)$ be a type II$_-$
functional). Set $i_q = w(f_q)$ for $q\in F$ and
\begin{eqnarray*}
E_0 &=& \{q: n \leqslant i_q <2^{2n}\}\\
E_1 &=& \{q: i_q < n\}\\
E_2 &=& \{q: 2^{2n} \leqslant i_q < n_1\}\\
J_k &=& \{q: n_k\leqslant i_q <
n_{k+1}\},\;\text{for}\;k<m\;\text{and}\quad J_m = \{q:
n_m\leqslant i_q\}
\end{eqnarray*}
\end{ntt}

Note that from Remark \ref{lacunarity} either $E_0 = \varnothing$
or $\#E_0 = 1$. Under the above notation the following lemma
holds, which is similar to Lemma 3.12 from \cite{AM} and their
proofs are almost identical.

\begin{lem}
Let $x = 2^n\sum_{k=1}^mc_kx_k$ be a $(C,\theta,n)$ exact vector in $\X$, where $\{x_k\}_{k=1}^m$ is $(C,\{n_k\}_{k=1}^m)\;\al$-RIS.

Then if $f = \frac{1}{2}\sum_{q\in F}(f_q + g_q)$ is a functional of type II$_+$, there exists $F_f\subset\{k: \ran f\cap\ran x_k\neq\varnothing\}$ with $\{\min\supp x_k: k\in F_f\} \in \mathcal{S}_2$ such that

\begin{eqnarray}
f(x) &<& \frac{1}{2}\sum_{q\in E_0}(f_q + g_q)(x) + \sum_{q\in E_1}
\frac{7C}{2^{i_q}} + \sum_{k=2}^m\sum_{q\in J_k}\frac{2^{n_k}}{2^{i_q + n_{k-1}}}\nonumber\\
&+&\sum_{k=1}^{m-1}\sum_{q\in J_k}\frac{C2^n}{2^{i_q}} + \sum_{q\in
E_2}\frac{C2^n}{2^{i_q}} + C2^n\sum_{k\in F_f}c_k
\end{eqnarray}

Similarly, if $f = \frac{1}{2}\sum_{q\in F}\la_q(f_q - g_q)$ is a functional of type II$_-$, there exists $F_f\subset\{k: \ran f\cap\ran x_k\neq\varnothing\}$ with $\{\min\supp x_k: k\in F_f\} \in \mathcal{S}_2$ such that

\begin{eqnarray}
f(x) &<& \frac{1}{2}\sum_{q\in E_0}\la_q(f_q - g_q)(x) + \sum_{q\in E_1}
\frac{7C}{2^{i_q}} + \sum_{k=2}^m\sum_{q\in J_k}\frac{2^{n_k}}{2^{i_q + n_{k-1}}}\nonumber\\
&+&\sum_{k=1}^{m-1}\sum_{q\in J_k}\frac{C2^n}{2^{i_q}} + \sum_{q\in
E_2}\frac{C2^n}{2^{i_q}} + C2^n\sum_{k\in F_f}c_k
\end{eqnarray}

\end{lem}

The next corollary is similar to Corollary 3.13 from \cite{AM}.

\begin{cor} Let $x$ be a $(C,\theta,n)$ exact vector in $\X$ and $f = \frac{1}{2}\sum_{q\in F}(f_q + g_q)$ be a type II$_+$
functional (or $f = \frac{1}{2}\sum_{q\in F}\la_q(f_q -
g_q)$ be a type II$_-$ functional), such that
$\{n,\ldots,2^{2n}\}\cap\hat{w}(f)=\varnothing$. Set $i_q = w(f_q)$ for $q\in F$ and $E_1
= \{q: i_q < n\}$. Then the following holds.

\begin{equation}
|f(x)| < \sum_{q\in E_1}\frac{7C}{2^{i_q}} + \frac{2C}{2^n}
\end{equation}
\label{corexactvectorroughestimate}
\end{cor}

The lemma which follows is similar to Lemma 3.14 from \cite{AM}.

\begin{lem}
Let $x = 2^n\sum_{k=1}^mc_kx_k$ be a $(C,\theta,n)$ exact vector
in $\X$ and $\be$ be a $\be$-average in $W$. Then there exists
$F_\be\subset\{k: \ran \be\cap\ran x_k\neq\varnothing\}$ with
$\{\min\supp x_k: k\in F_f\} \in \mathcal{S}_2$ such that

\begin{equation}
|\be(x)| < \frac{8C}{s(\be)} + C2^n\sum_{k\in F_\be}c_k
\end{equation}\label{exactvectorbestimate}
\end{lem}

The next lemma is similar to Lemma 3.15 from \cite{AM}.

\begin{lem}
Let $x$ be a $(C,\theta,n)$ exact vector in $\X$ and
$\{\be\}_{q=1}^d$ be a very fast growing and
$\Sj$-admissible sequence of $\be$-averages with
$j\leqslant n-3$. Then the following holds

\begin{equation}
\sum_{q=1}^d|\be_q(x)| < \sum_{q=1}^d\frac{8C}{s(\be_q)} +
\frac{1}{2^n}
\end{equation}
If moreover $s(\be_1)\geqslant \min\supp x$, then
$\sum_{q=1}^d|\be_q(x)| < \frac{2}{2^n}$

\end{lem}

The next result uses the previous lemma and it is similar to Proposition 3.16 from \cite{AM}.

\begin{cor}
Let $\{x_k\}_k$ be a block sequence in $\X$, such that $x_k$ is a
$(C,\theta,n_k)$ exact vector and $\{n_k\}_k$ is strictly
increasing. Then $\bdx = 0$.\label{corexactvectorbzero}
\end{cor}

\section{A combinatorial result}

In this section we introduce a new condition concerning the
behaviour of branches of special functionals on a block sequence
$\{x_k\}_k$ (see the definition below). When this condition is
satisfied, we shall write $\Bcrossx = 0$. We prove that one can
find in every block subspace a normalized block sequence
$\{x_k\}_k$ satisfying $\Bcrossx = 0$, as well as $\adx = 0$ and
$\bdx = 0$. We then proceed to prove a Ramsey type result
concerning block sequences with $\Bcrossx = 0$ and $\bdx = 0$. The
above are used in the next section to show that a block sequence
$\{x_k\}_k$ with $\Bcrossx = 0$, $\adx = 0$ and $\bdx = 0$, has a
subsequence generating a $c_0$ spreading model.

\begin{dfn}
Let $\{x_k\}_k$ be a block sequence in $\X$ and
$b = \{f_q,g_q\}_{q=1}^\infty\in\mathcal{B}$
(see the definition of the norming set) satisfying the following.
For every $\e>0$ there exist $k_0,q_0\inn$, such that for every $k\geqslant k_0,q\geqslant q_0$
we have that $|(f_q \substack{+\\[-2pt]-} g_q)(x_k)| < \e$. Then we write $\bcrossx = 0$.
If $\bcrossx = 0$ for every $b\in\mathcal{B}$, then we write $\Bcrossx = 0$.
\end{dfn}

\begin{rmk}
If $\bcrossx \neq 0$, using a pigeon hole argument, it is easy to
see that there exists  an infinite subset of the natural
numbers $M$ and $\e>0$ such that one of the following holds.

\begin{itemize}

\item[(i)] For every $k\in M$, there exists $q\inn$ such that $|(f_q + g_q)(x_k)| \geqslant \e$. In this case we say that $b_+\;\e$-norms $\{x_k\}_k$.

\item[(ii)] For every $k\in M$, there exists $q\inn$ such that $|(f_q - g_q)(x_k)| \geqslant \e$. In this case we say that $b_-\;\e$-norms $\{x_k\}_k$.

\end{itemize}
In either case we say that $b\;\e$-norms $\{x_k\}_k$.
\label{remarkseparate}
\end{rmk}

\begin{prp} Let $\{x_k\}_k$ be a bounded block sequence and
$b\in\mathcal{B}$ such that $b_+$\;$\e$-norms
$\{x_k\}_k$. Then there exists a subsequence of $\{x_k\}_k$ that
generates an $\ell_1$ spreading model.\label{bplusellone}
\end{prp}

\begin{proof}

If $b = \{f_q,g_q\}_{q=1}^\infty$ passing, if necessary,
to a subsequence, we may assume the following.

\begin{enumerate}

\item[(i)] For every $k\inn$ there exists $q_k\inn$ such that
$(f_{q_k} + g_{q_k})(x_k) > \e$ and $\min\supp
f_{q_k}\geqslant 2k$.

\item[(ii)] For $k\neq m\inn$, $\ran (f_{q_k} + g_{q_k})\cap
\ran x_m = \varnothing$

\end{enumerate}

Then for $n \leqslant k_1 <\cdots <k_n$ natural numbers and
$\{c_i\}_{i=1}^n$ non negative reals, we have that $f =
\frac{1}{2}\sum_{i=1}^n(f_{q_{k_i}} + g_{q_{k_i}})\in W$ and
$f(\sum_{i=1}^nc_ix_{k_i}) > \frac{\e}{2}\sum_{i=1}^nc_i$,
therefore
\begin{equation}
\|\sum_{i=1}^nc_ix_{k_i}\| >
\frac{\e}{2}\sum_{i=1}^nc_i\label{bplusselloneeq1}
\end{equation}

Since $\{x_k\}_k$ is weakly null, every spreading model admitted
by it must be unconditional. Combining this fact with
\eqref{bplusselloneeq1}, we conclude that every spreading model
admitted by $\{x_k\}_k$ is equivalent to the usual basis of
$\ell_1$.

\end{proof}

\begin{lem}
Let $\{x_k\}_k$ be a block sequence in $\X$ with $\bdx = 0$ and
$\e>0$. Then there exists  an infinite subset of the natural
numbers $M$, such that the set $B_\e = \{b\in\mathcal{B}: b$\;$\e$-norms\;$\{x_k\}_{k\in M}\}$ is finite.
\label{lemabzerobranchfinite}
\end{lem}

\begin{proof}
Towards a contradiction, assume that for every  infinite subset
of the natural numbers $M$, the set $\{b\in\mathcal{B}:
b\;\e$-norms\;$\{x_k\}_{k\in M}\}$ is infinite. By using
induction, choose  infinite subsets of the natural numbers $M_1\supset M_2\supset\cdots\supset
M_n\supset\cdots$ and
$\{b_n: n\inn\}\subset \mathcal{B}$ with $b_n\neq b_m$ for $n\neq
m$, satisfying the following. For every $n\inn$ and and $k\in
M_n$, if $b_n = \{f_q^n,g_q^n\}_{q=1}^\infty$ there exists
$q\inn$ such that either $|(f_q^n + g_q^n)(x_k)| > \e$ or
$|(f_q^n - g_q^n)(x_k)| > \e$. To simplify notation, from now
on we will assume that $|(f_q^n + g_q^n)(x_k)| > \e$.

We are going to prove the following. For every $k_0,m\inn$, there
exists $k\geqslant k_0$ and $\be$ a $\be$-average in $W$ of size
$s(\be) = m$, such that $\be(x_k) >\frac{\e}{2}$. By Proposition
\ref{bindexequivalent}, this means that $\bdx > 0$ which yields a
contradiction.

Let $k_0,m\inn$. Since $b_n\neq b_l$ for $n\neq l$, there exists
$q_0\inn$, such that for every $1\leqslant n<l\leqslant m$, for
every $q_1,q_2\geqslant q_0$, $w(f_{q_1}^n) \neq
w(f_{q_2}^l)$.

Choose $k\in M_m$ with $k\geqslant k_0$ and $\min\supp x_k\geq
\max\{\max\supp g_{q_0}^n: n=1,\ldots,m\}$. Then for
$n=1,\ldots,m$ there exists $q_n>q_0$ such that $|(f_{q_n}^n +
g_{q_n}^n)(x_k)| > \e$. Set $h_n = \sgn\big((f_{q_n}^n +
g_{q_n}^n)(x_k)\big)\frac{1}{2}(f_{q_n}^n + g_{q_n}^n)$ for
$n=1,\ldots,m$.

Then $h_n$ is a functional of type II in $W$ with $\hat{w}(h_n) =
\{w(f_{q_n}^n)\}$ for $n=1,\ldots,m$ and $h_n(x_k) >
\frac{\e}{2}$. Since $\hat{w}(h_n)\cap\hat{w}(h_l) = \varnothing$
for $1\leqslant n<l\leqslant m$, we have that $\be =
\frac{1}{m}\sum_{n=1}^mh_n$ is a $\be$-average of size $s(\be) =
m$ with $\be(x_k) > \frac{\e}{2}$. This completes the proof.

\end{proof}

\begin{lem}
Let $\{x_k\}_k$ be a block sequence in $\X$ with $\bdx = 0$. Then
there exists  an infinite subset of the natural numbers $M$, such
that the set $B = \{b\in\mathcal{B}:$\; there exists $\e>0$ such
that $b$\; $\e$-norms $\{x_k\}_{k\in M}\}$ is
countable. \label{lemabzerobranchcountable}
\end{lem}

\begin{proof}
Apply Lemma \ref{lemabzerobranchfinite} and choose  infinite subsets of the
natural numbers $M_1\supset
M_2\supset\cdots\supset M_n\supset\cdots$ such that the set $B_n = \{b\in\mathcal{B}: b$
\; $\frac{1}{n}$-norms $\{x_k\}_{k\in M_n}\}$ is finite,
for every $n\inn$. Choose $M$ a diagonalization of $\{M_n\}_n$.

We will show that $B = \{b\in\mathcal{B}:$\; there exists $\e>0$
such that $b$ \;$\e$-norms $\{x_k\}_{k\in
M}\}\subset\cup_nB_n$.

Let $b\in B$. Then, there exists $n\inn$, such that $b$
$\frac{1}{n}$-norms $\{x_k\}_{k\in M}$. It easily follows that
$b\in B_n$.

\end{proof}

\begin{lem}
Let $\{x_k\}_k$ be a bounded block sequence in $\X$ with $\bdx =
0$. Then there exists   an increasing sequence of
subsets of the natural numbers $\{F_k\}_k$ with $\#F_k\leqslant\min F_k$ for
all $k\inn$ with $\lim_k\#F_k = \infty$ such that if $y_k = \frac{1}{\#F_k}\sum_{i\in F_k}x_i$, then $\Bcrossy = 0$.
\label{lemabzeroaverages}
\end{lem}

\begin{proof}
Using Lemma \ref{lemabzerobranchcountable} and passing, if
necessary, to a subsequence, we may assume that if $B^\prime =
\{b\in\mathcal{B}:$\; there exists $\e>0$ such that $b$\;$\e$-norms $\{x_k\}_k\}$, then $B^\prime = \{b_n: n\inn\}$.

Let $b_n = \{f_q^n,g_q^n\}_{q=1}^\infty$ for all $n\inn$
and choose 
infinite subsets of the natural numbers $M_1\supset M_2\supset\cdots\supset M_n\supset\cdots$ such that for every
$n,q\inn$, there exists at most one $k\in M_n$, with $\ran
(f_q^n + g_q^n) \cap \ran x_k \neq \varnothing$.

Choose $M$ a diagonalization of $\{M_n\}_n$. Then for every
$n\inn$ there exists $q_n\inn$ such that for every $q\geqslant
q_n$ there exists at most one $k\in M$ with $\ran (f_q^n +
h_q^n) \cap \ran x_k \neq \varnothing$.

Choose  an increasing sequence of subsets of $M$ $\{F_k\}_k$ with $\#F_k\leqslant\min F_k$ for all $k\inn$ with
$\lim_k\#F_k = \infty$ and set $y_k = \frac{1}{\#F_k}\sum_{i\in
F_k}x_i$ for all $k\inn$.

Towards a contradiction, assume that there exist $\e > 0$
 and $b =
\{f_q,g_q\}_{q=1}^\infty\in \mathcal{B}$, such that $b$\;$\e$-norms $\{y_k\}_k$. For convenience, assume
that $b_+$ \;$\e$-norms $\{y_k\}_k$ and choose  an
infinite subset of the natural numbers $N$, such that for every $k\in N$
there exists $q_k\inn$ with $|(f_{q_k} + g_{q_k})(y_k)| > \e$.

It follows that for every $k\in N$, there exists $i_k\in F_k$ such
that $|(f_{q_k} + g_{q_k})(x_{i_k})| > \e$. We conclude that
$b$\;$\e$-norms $\{x_k\}_k$ and hence $b\in B^\prime$,
i.e. $b = b_n$, for some $n\inn$.

Choose $k\in N$ with $k > \max\supp g_{q_n}^n$ and $\#F_k >
\e^{-1}\sup\{\|x_k\|: k\inn\}$. Then for every $q\inn$, there
exists at most one $i\in F_k$, such that $\ran (f_{q}^n +
g_{q}^n) \cap \ran x_i \neq \varnothing$ and hence for every
$q\inn$, we have that $|(f_{q}^n + h_{q}^n)(y_k)| <
\frac{\sup\{\|x_k\|: k\inn\}}{\#F_k}<\e$. This contradiction
completes the proof.

\end{proof}

\begin{prp}
Let $\{x_k\}_k$ be a block sequence in $\X$ such that $x_k$ is a
$(C,\theta,n_k)$ exact vector with $n_k\in L_3$ (see the
definition of the coding function) and $\{n_k\}_k$ is strictly
increasing. Then $\Bcrossx = 0$. \label{exactvectorsbvoid}
\end{prp}

\begin{proof}

Let $b\in\mathcal{B}$
Observe that for $q\inn, h_q = \frac{1}{2}(f_{q} \plusminus g_{q})$
is a functional of type II and by Corollary
\ref{corexactvectorroughestimate}, if $i_q = w(f_{q})$ for
$k\inn$ we have that $|h_q(x_k)| < \frac{7C}{2^{i_q}} +
\frac{2C}{2^{n_k}}$. From this it easily follows that $\bcrossx = 0$.
\end{proof}

\begin{prp}
Let $\{x_k\}_k$ be a normalized block sequence in $\X$. Then there
exists $\{y_k\}_k$ a further normalized block sequence of
$\{x_k\}_k$ such that $\ady = 0, \bdy = 0$ and $\Bcrossy = 0$.
\label{propexistence}
\end{prp}

\begin{proof}
Since $\X$ does not contain a copy of $c_0$, we may choose $\{z_k\}_k$ a
normalized block sequence of $\{x_k\}_k$, such that if $z_k =
\sum_{i\in G_k}c_ix_i$, then $\lim_k\max\{|c_i|: i\in G_k\} = 0$.

If $\adz = 0, \bdz = 0$ and $\Bcrossz = 0$, then $\{z_k\}_k$ is the desired sequence. Otherwise,
we distinguish three cases.\vskip6pt

\noindent{\em Case 1:} $\adz = 0, \bdz = 0$ and there exist
$b\in\mathcal{B}$, $\e>0$ such that $b_+$ \;$\e$-norms
$\{z_k\}_k$.

Using Proposition \ref{bplusellone} and passing, if necessary, to a
subsequence, we may assume that $\{z_k\}_k$ generates an $\ell_1$
spreading model. Apply Lemma \ref{lemabzeroaverages} to find
 an increasing sequence of subsets of the natural
numbers $\{F_k\}_k$ with $\#F_k\leqslant\min F_k$ for all $k\inn$ with
$\lim_k\#F_k = \infty$ such that if $y_k =
\frac{1}{\#F_k}\sum_{i\in F_k}z_i$, then $\Bcrossy = 0$.

Since $\{z_k\}_k$ generates an $\ell_1$ spreading model, we have
that $\{y_k\}_k$ is seminormalized. Moreover Remark
\ref{indexconvex} yields that $\ady = 0$ as well as $\bdy = 0$. We
conclude that if $y_k^\prime = \frac{1}{\|y_k\|}y_k$, then
$\{y_k^\prime\}_k$ is the desired sequence.\vskip6pt

\noindent{\em Case 2:} $\adz = 0, \bdz = 0$ and there exist
$b\in\mathcal{B}$, $\e>0$ such that $b_-$ \;$\e$-norms
$\{z_k\}_k$.

If $b = \{f_{q},g_{q}\}_{q=1}^\infty$ passing if necessary to
a subsequence, we may assume that for every $k\inn$ there exists
$q_k\inn$ such that $|(f_{q_k} - g_{q_k})(z_k)| > \e$ and
$\max\{|c_i|: i\in F_k\} < \frac{\e}{2}$.

Fix $k\inn$ and set $i_k = \max\{i\in G_k: \ran f_{q_k}\cap\ran x_i\neq\varnothing\}, G_k^1 = \{i\in G_k: i\leqslant
i_k\}$ and $G_k^2 = \{i\in G_k: i> i_k\}$. Set
\begin{equation*}
z_k^\prime = \sgn\big(f_{q_k}(z_k)\big)\sum_{i\in
G_k^1}c_ix_i + \sgn\big(g_{q_k}(z_k)\big)\sum_{i\in G_k^2}c_ix_i
\end{equation*}
Observe the following.
\begin{eqnarray*}
f_{q_k}(z_k^\prime) &=& |f_{q_k}(z_k)|\\
g_{q_k}(z_k^\prime) &>& |g_{q_k}(z_k)| - |c_{i_k}| >
|g_{q_k}(z_k)| - \frac{\e}{2}
\end{eqnarray*}
\begin{equation*}
\frac{1}{2}\leqslant\|z_k^\prime\|\leqslant 2
\end{equation*}
Combining the above we conclude that by setting $w_k =
\frac{1}{\|z_k^\prime\|}z_k^\prime$, we have that $(f_{q_k} +
g_{q_k})(w_k)
> \frac{\e}{4}$, i.e. $b_+$\; $\frac{\e}{4}$-norms
$\{w_k\}_k$. Moreover Remarks \ref{indexinterval} and
\ref{indexsum} yield that $\adw = 0$ as well as $\bdw = 0$, hence
this case has been reduced to the previous one.\vskip6pt

\noindent{\em Case 3:} $\adz > 0$ or $\bdz > 0$.

Apply Proposition \ref{ell1nandexactvectorexist} to construct a
sequence  of $(C,\theta,n_k)$ vectors $\{y_k^\prime\}_k$ with
$\{n_k\}_k$ strictly increasing. Set $y_k =
\frac{1}{\|y_k\|}y_k^\prime$. Corollary \ref{corexactvetorazero}
yields that $\ady = 0$ and passing, if necessary to a subsequence,
$\{y_k\}_k$ is $(C,\{n_k\}_k)\;\al$-RIS.

Assume that $\bdy = 0$. Then this case is reduced either to case
1, or to case 2.

If on the other hand $\bdy > 0$, apply Proposition
\ref{ell1nandexactvectorexist} to construct a sequence
 of $(C,\theta,n_k)$ exact vectors $\{w_k^\prime\}_k$ with $n_k\in
L_3$ for all $k\inn$ and $\{n_k\}_k$ strictly increasing. Set $w_k
= \frac{1}{\|w^\prime_k\|}w_k$. Corollaries
\ref{corexactvetorazero}, \ref{corexactvectorbzero} and
Proposition \ref{exactvectorsbvoid} yield that $\{w_k\}_k$ is the
desired sequence.

\end{proof}

The following definition is a slight variation of Definition 4.1 from \cite{AM}.

\begin{dfn}
Let $x_1<x_2<x_3$ be vectors in $\X$, $f =
\substack{+\\[-2pt]-}E\big(\frac{1}{2}\sum_{q\in F}(f_q + g_q)\big)$ be a
functional of type II$_+$ ( or $f = \substack{+\\[-2pt]-}E\big(\frac{1}{2}\sum_{q\in
F}\la_q(f_q - g_q)\big)$ be a functional of type II$_-$), such
that $\supp f\cap\ran x_i\neq\varnothing$, for $i=1,2,3$. Set $q_0
= \min\{q\in F: \ran(f_q + g_q) \cap \ran x_2\neq\varnothing\}$.
If $\ran (f_{q_0} + g_{q_0}) \cap \ran x_3 = \varnothing$, then we
say that $f$ separates $x_1,x_2,x_3$.
\end{dfn}

\begin{lem}
Let $\{n_k\}_k$ be a strictly increasing sequence of natural
numbers satisfying the following. For every $m\inn$, there exists
a special sequence $\{f_{q}^m,g_{q}^m\}_{q=1}^{d_m}$ such that
$\{n_k: k=1,\ldots,m\} \subset \{w(f_{q}^m): q=1,\ldots,d_m\}$.
Then there exists $b =
\{f_{q},g_{q}\}_{q=1}^\infty\in\mathcal{B}$, such that $\{n_k:
k\inn\} \subset \{w(f_{q}): q\inn\}$.\label{weightsb}
\end{lem}

\begin{proof}
We construct $b$ by induction. Let $m\inn$ and suppose that we
have chosen natural numbers $1\leqslant p_1 <\cdots<p_m$ and a
special sequence $\{f_{q},g_{q}\}_{q=1}^{p_m}$ such that the
following are satisfied. For $1\leqslant l\leqslant m$
\begin{itemize}

\item[(i)] $\{n_k: k=1,\ldots,l\} \subset \{w(f_{q}):
q=1,\ldots,p_l\}$

\item[(ii)] $\sigma(f_1,g_1,f_2,g_2\ldots,f_{p_l},g_{p_l}) = n_{l+1}$

\end{itemize}

Since $\{n_k: k =1,\ldots,m+2\}\subset \{w(f_{q}^{m+2}):
q=1,\ldots,d_{m+2}\}$, there exist $1 < q_0 < q_1\leqslant
d_{m+2}$, such that $w(f_{q_0}^{m+2}) = n_{m+1}$ and
$w(f_{q_1}^{m+2}) = n_{m+2}$

Then
\begin{equation*}
\sigma(f_1^{m+2},g_1^{m+2},\ldots,f_{q_1-1}^{m+2},g_{q_1-1}^{m+2})
= n_{m+2}
\end{equation*}

Set $p_{m+1} = q_1 - 1$. It remains to be shown that $p_m <
p_{m+1}$ and that $\{f_{q},g_{q}\}_{q=1}^{p_m} =
\{f_{q}^{m+2},g_{q}^{m+2}\}_{q=1}^{p_m}$. Then, $\{f_q, g_q\}_{q=1}^{p_{m+1}} = \{f_{q}^{m+2},g_{q}^{m+2}\}_{q=1}^{p_{m+1}}$ will be the desired special sequence.

Since
\begin{eqnarray*}
n_{m+1} &=& \sigma(f_1,g_1,\ldots,f_{p_m},g_{p_m})\\
w(f^{m+2}_{q_0}) &=&
\sigma(f_1^{m+2},g_1^{m+2},\ldots,f_{q_0-1}^{m+2},g_{q_0-1}^{m+2})
\end{eqnarray*}
and $w(f^{m+2}_{q_0}) = n_{m+1}$, by the fact that $\sigma$ is
one to one, we conclude that $\{f_{q},g_{q}\}_{q=1}^{p_m} =
\{f_{q}^{m+2},g_{q}^{m+2}\}_{q=1}^{q_0-1}$. Thus, it follows
that $p_m = q_0-1 < q_1 - 1 = p_{m+1}$ and
$\{f_{q},g_{q}\}_{q=1}^{p_m} =
\{f_{q}^{m+2},g_{q}^{m+2}\}_{q=1}^{p_m}$.

\end{proof}

\begin{prp}
Let $\{x_k\}_k$ be a bounded block sequence in $\X$, such that
$\bdx = 0$ and $\Bcrossx = 0$. Then for any $\e>0$, there exists  an infinite
subset of the natural numbers $M$, such that for any $k_1 < k_2 < k_3\in M$,
for any functional $f\in W$ of type II that separates
$x_{k_1},x_{k_2},x_{k_3}$, we have that $|f(x_{k_i})| < \e$, for
some $i\in\{1,2,3\}$.\label{propcombinatorial}
\end{prp}

\begin{proof}
Towards a contradiction, assume that this is not the case. By
using Ramsey theorem \cite{Ra}, we may assume that there exists
$\e>0$ such that for any $k<l<m\inn$, there exists  a
functional of type II $f_{k,l,m}$ that separates $x_k,x_l,x_m$ and
$|f_{k,l,m}(x_k)|>\e, |f_{k,l,m}(x_l)|>\e, |f_{k,l,m}(x_m)|>\e$.
We may also assume that $f_{k,l,m}$ is of type II$_+$, for every
$k<l<m\inn$, or that $f_{k,l,m}$ is of type II$_-$, for every
$k<l<m\inn$. From now on we shall assume the first.

For $1<k<m\inn$, there exists $b_{k,m} =
\{f_{q}^{k,m},g_{q}^{k,m}\}_{q=1}^\infty\in\mathcal{B}$ and  intervals of the natural numbers $E_{k,m}$, with
$f_{1,k,m} = E_{k,m}\bigg(\frac{1}{2}\sum_{q\in
F_{k,m}}(f_{q}^{k,m} + g_{q}^{k,m})\bigg)$. Set

\begin{eqnarray*}
p_{k,m} &=& \min\{q\in F_{k,m}: \ran(f_{q}^{k,m} +
g_{q}^{k,m})\cap x_1\neq\varnothing\}\\
q_{k,m} &=& \min\{q\in F_{k,m}: \ran(f_{q}^{k,m} +
g_{q}^{k,m})\cap x_k\neq\varnothing\}
\end{eqnarray*}

Notice, that for $1<k<m$, since $|f_{1,k,m}(x_1)| > \e$, it
follows that, if $w(f_{p_{k,m}}^{k,m}) = j_{k,m}$
\begin{equation*}
\frac{1}{2^{j_{k,m}}} > \frac{\e}{\|x_1\|\max\supp x_1}
\end{equation*}

By applying Ramsey theorem once more, we may assume that there
exists $j_1\inn$, such that for any $1<k<m$, we have that
$w(f_{p_{k,m}}^{k,m}) = j_1$.

Arguing in the same way and diagonalizing, we may assume that for
any $k>1$, there exists $j_k\inn$ such that for any $m>k$, we have
that $w(f^{k,m}_{q_{k,m}}) = j_k$.

Moreover, for every $1<k<m\inn$, the following holds.
\begin{equation*}
2(\#F_{k,m}) \leqslant \min\supp f_{p_{k,m}}^{k,m}\leqslant
\max\supp x_1
\end{equation*}
Setting $\e^\prime = \frac{4\e}{\max\supp x_1}$, there exists
$r_{k,m}\in F_{k,m}$ such that
\begin{equation}
|E_{k,m}\big(\frac{1}{2}(f_{r_{k,m}}^{k,m} +
g_{r_{k,m}}^{k,m})\big)(x_m)| > \e^\prime
\end{equation}
Since $f_{1,k,m}$ separates $x_1,x_k,x_m$, it follows that
$r_{k,m} > q_{k,m}$.

Set $i_{k,m} = w(f_{r_{k,m}}^{k,m})$ for all $1<k<m\inn$ and
\begin{equation*}
A = \big\{\{k,l,m\}\in[\mathbb{N}\setminus\{1\}]^3: i_{k,m} =
i_{l,m}\big\}
\end{equation*}

Applying Ramsey theorem once more, we may assume that either
$[\mathbb{N}\setminus\{1\}]^3\subset A$ or
$[\mathbb{N}\setminus\{1\}]^3\subset A^c$.

Assume that $[\mathbb{N}\setminus\{1\}]^3\subset A^c$. Then, for
$m>2$, we have that
\begin{equation*}
h_k = \sgn\bigg(E_{k,m}\big(\frac{1}{2}(f_{r_{k,m}}^{k,m} +
g_{r_{k,m}}^{k,m})\big)(x_m)\bigg)E_{k,m}\big(\frac{1}{2}(f_{r_{k,m}}^{k,m}
+ g_{r_{k,m}}^{k,m})\big)
\end{equation*}
are functionals of type II with pairwise disjoint weights
$\hat{w}(h_k)$ and $h_k(x_m) > \e^\prime$ for $k=2,\ldots,m-1$. We
conclude that $\be = \frac{1}{m-2}\sum_{k=2}^{m-1}h_k$ is a
$\be$-average in $W$ of size $s(\be) = m-2$ and $\be(x_m) >
\e^\prime$. Proposition \ref{bindexequivalent} yields that $\bdx >
0$, which is absurd.

Hence, we may assume that $[\mathbb{N}\setminus\{1\}]^3\subset A$,
i.e. for every $m>2$, there exists $i_m\inn$, such that for every
$1<k<m$, $i_{k,m} = i_m$. By the fact that $\sigma$ is one to one,
we conclude that for every $m>2$, by setting
$\{f_{q}^m,g_{q}^m\}_{q=1}^{r_m} = \sigma^{-1}(\{i_m\})$ the
following holds.
\begin{equation}
\{f_{q}^{k,m},g_{q}^{k,m}\}_{q=1}^{r_{k,m}-1} =
\{f_{q}^m,g_{q}^m\}_{q=1}^{r_m},\quad\text{for}\;1<k<m\label{combinatorialeq1}
\end{equation}

Set
\begin{equation*}
C = \big\{\{k,l\}\in[\mathbb{N}\setminus\{1\}]^2: j_k\neq
j_l\big\}
\end{equation*}

Applying Ramsey theorem once more, we may assume that either
$[\mathbb{N}\setminus\{1\}]^2\subset C$ or
$[\mathbb{N}\setminus\{1\}]^2\subset C^c$.

Assume that $[\mathbb{N}\setminus\{1\}]^2\subset C^c$. Then there
exists $j_0\inn$, such that $j_k = j_0$ for all $k>1$. For
$1<k<m$, by \eqref{combinatorialeq1}
$\{f_{q_{k,m}}^{k,m},g_{q_{k,m}}^{k,m}\}\in \{f_{q}^m,g_{q}^m:
q=1,\ldots,r_m\}$. Since for $2<k<m$, $j_2 = j_k$, we conclude
that $\{f_{q_{2,m}}^{2,m},g_{q_{2,m}}^{2,m}\} =
\{f_{q_{k,m}}^{k,m},g_{q_{k,m}}^{k,m}\}$.

Set $h_m = \frac{1}{2}(f_{q_{2,m}}^{2,m} + g_{q_{2,m}}^{2,m})$. By
the fact that $f_{2,m}, f_{m-1,m}$ separate $x_1,x_2,x_m$ and
$x_1,x_{m-1},x_m$ respectively, we conclude that $\ran
x_k\subset\ran h_m$ and $|h_m(x_k)| > \e$ for $k=3,\ldots,m-2$.
Choose $h$ a $w^*$-limit point of $\{h_m\}_m$. Then
$|h(x_k)|\geqslant\e$ for every $k>2$. Corollary \ref{shrinking}
yields a contradiction.

Hence, we may assume that $[\mathbb{N}\setminus\{1\}]^2\subset C$,
and that $\{j_k\}_k$ is strictly increasing. Lemma \ref{weightsb}
and \eqref{combinatorialeq1} yield that there exists $b =
\{f_q,g_q\}_{q=1}^\infty\in\mathcal{B}$, such that $\{j_k:
k\inn\}\subset\{w(f_{q}): q\inn\}$.

We will show that $b$ $\e^\prime$-norms $\{x_k\}_k$, which will
complete the proof. Let $1<k<m\inn$. Arguing as previously, there
exists $t_{k,m}\in F_{k,m}$, such that
$|(f_{t_{k,m}}^{k,m}+g_{t_{k,m}}^{k,m})(x_k)|
> \e^\prime$. Evidently, $q_{k,m}\leqslant t_{k,m} \leqslant r_{k,m}$ Set
\begin{equation*}
D = \{\{k,m\}\in [\mathbb{N}\setminus\{1\}]^2: t_{k,m} < r_{k,m}\}
\end{equation*}

Applying Ramsey theorem one last time, we may assume that either
$[\mathbb{N}\setminus\{1\}]^2\subset D$, or
$[\mathbb{N}\setminus\{1\}]^2\subset D^c$.

If $[\mathbb{N}\setminus\{1\}]^2\subset D^c$, then for $m>3$, by
 \eqref{combinatorialeq1} we have that $t_{m-2,m} = r_{m-2,m} =
r_m + 1$ and
$\{f_1^m,g_1^m,\ldots,f_{r_m}^m,g_{r_m}^m,f_{t_{m-2,m}}^{m-2,m},g_{t_{m-2,m}}^{m-2,m}\}$
is a special sequence.

Similarly, by \eqref{combinatorialeq1} we have that $t_{m-1,m} =
r_{m-1,m} = r_m + 1$ and that
$\{f_1^m,g_1^m,\ldots,f_{r_m}^m,g_{r_m}^m,f_{t_{m-1,m}}^{m-1,m},g_{t_{m-1,m}}^{m-1,m}\}$
is a special sequence.

Since $q_{m-1,m}<r_{m-1,m} = t_{m-1,m}$, we have that there exists
$q\leqslant r_m$, such that
$\{f_{q_{m-1,m}}^{m-1,m},g_{q_{m-1,m}}^{m-1,m}\} = \{f_q^m,
g_q^m\}$.

This means the following.
\begin{eqnarray*}
\max\supp x_{m-2} &<& \min\supp x_{m-1} \leqslant \max\supp
g_{q_{m-1,m}}\\ &=& \max\supp g_{q}^m < \min\supp
f_{t_{m-2,m}}^{m-2,m}
\end{eqnarray*}
We conclude that $\ran(f_{t_{m-2,m}}^{m-2,m} +
g_{t_{m-2,m}}^{m-2,m})\cap\ran x_{m-2} = \varnothing$. This cannot
be the case and hence we conclude that
$[\mathbb{N}\setminus\{1\}]^2\subset D$.

Let $k\inn$. We will show that
$f_{t_{k,k+3}}^{k,k+3}+g_{t_{k,k+3}}^{k,k+3}\in b_+$. First,
observe that by \eqref{combinatorialeq1} and the fact that
$t_{k,k+3} \leqslant r_{k,k+3} - 1 =  r_{k+3}$, we have that
\begin{eqnarray*}
(f_{t_{k,k+3}}^{k,k+3}+g_{t_{k,k+3}}^{k,k+3})&\in&\{f_{q}^{k+3}+g_{q}^{k+3}:
q=1,\ldots,r_{k+3}\}\\
(f_{q_{k+1,k+3}}^{k+1,k+3}+g_{q_{k+1,k+3}}^{k+1,k+3})&\in&\{f_{q}^{k+3}+g_{q}^{k+3}:
q=1,\ldots,r_{k+3}\}\\
(f_{q_{k+2,k+3}}^{k+2,k+3}+g_{q_{k+2,k+3}}^{k+2,k+3})
&\in&\{f_{q}^{k+3}+g_{q}^{k+3}: q=1,\ldots,r_{k+3}\}
\end{eqnarray*}
Thus, we moreover have that
\begin{equation*}
(f_{t_{k,k+3}}^{k,k+3}+g_{t_{k,k+3}}^{k,k+3})\leqslant
(f_{q_{k+1}-1}^{k+1,k+3}+g_{q_{k+1}}^{k+1,k+3})<
(f_{q_{k+2}}^{k+2,k+3}+g_{q_{k+2}}^{k+2,k+3})
\end{equation*}

By the fact that $\sigma$ is one to one, we conclude that
$\{f_{t_{k,k+3}}^{k,k+3},g_{t_{k,k+3}}^{k,k+3}\}\in\sigma^{-1}(\{j_{k+2}\})\subset\big\{\{f_q,g_q\}:q\inn\big\}$.

\end{proof}

\section{$c_0$ spreading models}

In this section we prove that a sequence $\{x_k\}_k$ satisfying
$\Bcrossx = 0$, $\adx = 0$ as well as $\bdx = 0$ has a subsequence
generating a $c_0$ spreading model. This is crucial, as a
spreading model universal sequence is constructed on a sequence
generating a $c_0$ spreading model.

\begin{prp} Let $x_1<\cdots<x_n$ be a seminormalized block sequence in $\X$,
such that $\|x_k\|\leqslant 1$ for $k=1,\ldots,n, n\geqslant 3$
and there exist $n+3\leqslant j_1 <\cdots< j_n$ strictly
increasing natural numbers, such that the following are satisfied.
\begin{enumerate}

\item[(i)] For any $k_0\in\{1,\ldots,n\}$, for any $k\geqslant
k_0, k\in\{1,\ldots,n\}$, for any $\{\al_q\}_{q=1}^d$ very fast
growing and $\Sj$-admissible sequence of $\al$-averages, with
$j<j_{k_0}$ and $s(\al_1) > \min\supp x_{k_0}$, we have that
$\sum_{q=1}^d|\al_q(x_k)| < \frac{1}{n\cdot2^n}$.

\item[(ii)] For any $k_0\in\{1,\ldots,n\}$, for any $k\geqslant
k_0, k\in\{1,\ldots,n\}$, for any $\{\be_q\}_{q=1}^d$ very fast
growing and $\Sj$-admissible sequence of $\be$-averages, with
$j<j_{k_0}$ and $s(\be_1) > \min\supp x_{k_0}$, we have that
$\sum_{q=1}^d|\be_q(x_k)| < \frac{1}{n\cdot2^n}$.

\item[(iii)] For $k = 1,\ldots,n-1$, the following holds:
$\frac{1}{2^{j_{k+1}}}\max\supp x_k < \frac{1}{2^n}$.

\item[(iv)] For any $1\leqslant k_1 < k_2 < k_3\leqslant n$, for
any functional $f\in W$ of type II that separates
$x_{k_1},x_{k_2},x_{k_3}$, we have that $|f(x_{k_i})| <
\frac{1}{n\cdot2^n}$, for some $i\in\{1,2,3\}$.

\end{enumerate}

Then $\{x_k\}_{k=1}^n$ is equivalent to the unit vector basis of $\ell_\infty^n$,
with an upper constant $4 + \frac{5}{2^n}$. Moreover, for any
functional $f\in W$ of type I$_\al$ with weight $w(f) = j < j_1$,
we have that $|f(\sum_{k=1}^nx_k)| < \frac{4 +
\frac{6}{2^n}}{2^j}$.\label{finitec0}
\end{prp}

\begin{proof}
As in the proof of Proposition 4.7 from \cite{AM}, we will
inductively prove, that for any $\{c_k\}_{k=1}^n\subset[-1,1]$ the
following hold.
\begin{enumerate}

\item[(i)] For any $f\in W$, we have that $|f(\sum_{k=1}^nc_kx_k)|
< (4 + \frac{5}{2^n})\max\{|c_k|:k=1,\ldots,n\}$.

\item[(ii)] If $f$ is of type I$_\al$ and $w(f)\geqslant 3$, then
$|f(\sum_{k=1}^nc_kx_k)| < (1 +
\frac{2}{2^n})\max\{|c_k|:k=1,\ldots,n\}$.

\item[(iii)] If $f$ is of type I$_\al$ and $w(f) = j < j_1$, then
$|f(\sum_{k=1}^nc_kx_k)| < \frac{4 +
\frac{6}{2^n}}{2^j}\max\{|c_k|:k=1,\ldots,n\}$.

\end{enumerate}

For any functional $f\in W_0$ the inductive assumption holds.
Assume that it holds for any $f\in W_m$ and let $f\in W_{m+1}$. If
$f$ is a convex combination, then there is nothing to prove.

Assume that $f$ is of type I$_\al, f =
\frac{1}{2^j}\sum_{q=1}^d\al_q$, where $\{\al_q\}_{q=1}^d$ is a
very fast growing and $\mathcal{S}_j$-admissible sequence of
$\al$-averages in $W_m$.

Set $k_1 = \min\{k: \ran f\cap\ran x_k\neq\varnothing\}$ and $q_1
= \min\{q: \ran \al_q\cap\ran x_{k_1}\neq\varnothing\}$.

We distinguish 3 cases.\vskip3pt

\noindent {\em Case 1:} $j<j_1$.

For $q>q_1$, we have that $s(\al_q) > \min\supp x_{k_1}$,
therefore we conclude that
\begin{equation}
\sum_{q>q_1}|\al_q(\sum_{k=1}^nc_kx_k)| <
\frac{1}{2^n}\max\{|c_k|:k=1,\ldots,n\} \label{prop4.6eq1}
\end{equation}
while the inductive assumption yields that
\begin{equation}
|\al_{q_1}(\sum_{k=1}^nc_kx_k)| < (4 +
\frac{5}{2^n})\max\{|c_k|:k=1,\ldots,n\} \label{prop4.6eq2}
\end{equation}
Then \eqref{prop4.6eq1} and \eqref{prop4.6eq2} allow us to
conclude that
\begin{equation}
|f(\sum_{k=1}^nc_kx_k)| < \frac{4 +
\frac{6}{2^n}}{2^j}\max\{|c_k|:k=1,\ldots,n\}\label{prop4.6eqx1}
\end{equation}
Hence, (iii) from the inductive assumption is satisfied.\vskip3pt

\noindent {\em Case 2:} There exists $k_0 < n$, such that
$j_{k_0}\leqslant j < j_{k_0 + 1}$.

Arguing as previously we get that
\begin{equation}
|f(\sum_{k>k_0}c_kx_k)| < \frac{4 +
\frac{6}{2^n}}{2^{j_{k_0}}}\max\{|c_k|:k=1,\ldots,n\} <
\frac{1}{2^n}\max\{|c_k|:k=1,\ldots,n\} \label{prop4.6eq3}
\end{equation}
and
\begin{equation}
|f(\sum_{k<k_0}c_kx_k)| < \frac{1}{2^n}\max\{|c_k|:k=1,\ldots,n\}
\label{prop4.6eq4}
\end{equation}
Using \eqref{prop4.6eq3}, \eqref{prop4.6eq4}, the fact that
$|f(x_{k_0})| \leqslant 1$, we conclude that
\begin{equation}
|f(\sum_{k=1}^nc_kx_k)| < (1 +
\frac{2}{2^n})\max\{|c_k|:k=1,\ldots,n\}\label{prop4.6eqx2}
\end{equation}

\noindent {\em Case 3:} $j\geqslant j_n$

By using the same arguments, we conclude that
\begin{equation}
|f(\sum_{k=1}^nc_kx_k)| < (1 +
\frac{1}{2^n})\max\{|c_k|:k=1,\ldots,n\}\label{prop4.6eqx3}
\end{equation}

Then \eqref{prop4.6eqx1}, \eqref{prop4.6eqx2} and
\eqref{prop4.6eqx3} yield that (ii) from the inductive assumption
is satisfied.

If $f$ is of type I$_\be$, then the proof is exactly the same,
therefore assume that $f$ is of type II$_+$, $f =
\frac{1}{2}\sum_{q\in F}^d(f_q + g_q)$, where $\{f_q,g_q\}_{q\in
F}$ are functionals of type I$_\al$. Set
\begin{eqnarray*}
E &=& \{k: |f(x_k)| \geqslant \frac{1}{n\cdot2^n}\}\\
E_1 &=& \{k\in E:\;\text{there exist at least two}\;q\;\text{such
that}\;\ran (f_q+g_q)\cap\ran x_k\neq\varnothing\}
\end{eqnarray*}
Then $\#E_1\leqslant 2$. Indeed, if $k_1<k_2<k_3\in E_1$, then $f$
separates $x_{k_1},x_{k_2}$ and $x_{k_3}$ which contradicts our
initial assumptions.

If moreover we set $J = \{q:$ there exists $k\in E\setminus E_1$
such that $\ran (f_q+g_q)\cap\ran x_k\neq\varnothing\}$, then for
the same reasons we get that $\#J\leqslant 2$.

Since for any $j$, we have that $w(f_q),w(g_q)\in L_0$, we get
that $w(f_j) > 9$, therefore:
\begin{eqnarray}
|f(\sum_{k \in E\setminus E_1}^nc_kx_k)| &<& (2 + \frac{4}{2^n})\max\{|c_k|:k=1,\ldots,n\}\label{prop4.6eq5}\\
|f(\sum_{k \in  E_1}^nc_kx_k)| &\leqslant& 2 \max\{|c_k|:k=1,\ldots,n\}\label{prop4.6eq6}\\
|f(\sum_{k \notin  E}^nc_kx_k)| &\leqslant&
n\cdot\frac{1}{n\cdot2^n}\max\{|c_k|:k=1,\ldots,n\}\label{prop4.6eq7}
\end{eqnarray}
Finally, \eqref{prop4.6eq5} to \eqref{prop4.6eq7} yield the
following.
\begin{equation*}
|f(\sum_{k=1}^nc_kx_k)| < (4 +
\frac{5}{2^n})\max\{|c_k|:k=1,\ldots,n\}
\end{equation*}

If $f$ is of type II$_-$, the proof is exactly the same. This
means that (i) from the inductive assumption is satisfied an this
completes the proof.

\end{proof}

\begin{prp}
Let $\{x_k\}_{k\inn}$ be a seminormalized block sequence in $\X$,
such that $\|x_k\|\leqslant 1$ for all $k\inn$, $\adx = 0$ as well
as $\bdx = 0$ and $\Bcrossx = 0$. Then it has a subsequence, again
denoted by $\{x_k\}_{k\inn}$ satisfying the following.
\begin{enumerate}

\item[(i)] $\{x_k\}_{k\inn}$ generates a $c_0$ spreading model.
More precisely, for any $n\leqslant k_1<\cdots<k_n$, we have that
$\|\sum_{i=1}^nx_{k_i}\| \leqslant 5$.

\item[(i)] There exists a strictly increasing sequence of natural numbers
$\{j_n\}_{n\inn}$, such that for any $n\leqslant k_1 < \cdots
<k_n$, for any functional $f$ of type I$_\al$ with $w(f) = j<j_n$,
we have that
    \begin{equation*}
    |f(\sum_{i=1}^nx_{k_i})| < \frac{5}{2^j}
    \end{equation*}
\end{enumerate}\label{c0spreadingmodel}
\end{prp}

\begin{proof}
By repeatedly applying Proposition \ref{propcombinatorial} and
diagonalizing, we may assume that for any $n\leqslant k_1 < k_2 <
k_3$, for any functional $f$ of type II that separates
$x_{k_1},x_{k_2}$ and $x_{k_3}$, we have that $|f(x_{k_i})| <
\frac{1}{n\cdot2^n}$, for some $i\in\{1,2,3\}$.

Use Propositions \ref{aindexequivalent} and \ref{bindexequivalent} to inductively
choose a subsequence of $\{x_k\}_{k\inn}$, again denoted by
$\{x_k\}_{k\inn}$ and  a strictly increasing
sequence of natural numbers $\{j_k\}_{k\inn}$ with $j_k\geqslant k+3$ for all $k\inn$, such
that the following are satisfied.
\begin{enumerate}

\item[(i)] For any $k_0\inn$, for any $k\geqslant k_0$, for any
$\{\al_q\}_{q=1}^d$ very fast growing and
$\mathcal{S}_j$-admissible sequence of $\al$-averages, with
$j<j_{k_0}$ and $s(\al_1) > \min\supp x_{k_0}$, we have that
$\sum_{q=1}^d|\al_q(x_k)| < \frac{1}{k_0\cdot2^{k_0}}$.

\item[(ii)] For any $k_0\inn$, for any $k\geqslant k_0$, for any
$\{\be_q\}_{q=1}^d$ very fast growing and
$\mathcal{S}_j$-admissible sequence of $\be$-averages, with
$j<j_{k_0}$ and $s(\be_1) > \min\supp x_{k_0}$, we have that
$\sum_{q=1}^d|\be_q(x_k)| < \frac{1}{k_0\cdot2^{k_0}}$.

\item[(iii)] For $k\inn$, the following holds:
$\frac{1}{2^{j_{k+1}}}\max\supp x_k < \frac{1}{2^k}$.
\end{enumerate}
It is easy to check that for $n\leqslant k_1<\cdots<k_n$, the
assumptions of Proposition \ref{finitec0} are satisfied.

\end{proof}

\section{Spreading model universal block sequences}

In this section we define exact pairs and exact nodes in $\X$.
Then, using a sequence generating a $c_0$ spreading model, we pass
to a sequence of exact nodes $\{x_k,y_k,f_k,g_k\}$, such that
$\{f_k,g_k\}_{k=1}^\infty$ defines a special branch. Setting $z_k
= x_k -y_k$, we prove that $\{z_k\}_k$ is a spreading model
universal sequence. Using the structure of such sequences, we also
prove that the space $\X$ is hereditarily indecomposable.

\begin{dfn}
A pair $\{x,f\}$, where $x\in\X, f\in W$ is called an $n$-exact
pair if the following hold.

\begin{itemize}

\item[(i)] $f$ is a functional of type I$_\al$ with $w(f) = n$,
$\min\supp x\leqslant \min\supp f$ and $\max\supp
x\leqslant\max\supp f$.

\item[(ii)] There exists  a $(5,1,n)$ exact vector $x^\prime \in\X$
such that $1\geqslant f(x^\prime) > \frac{35}{36}$ and $x =
\frac{x^\prime}{f(x^\prime)}$.

\end{itemize}

\end{dfn}

\begin{rmk}
If $\{x,f\}$ is a $n$-exact pair, then $f(x) = 1$ and by Remark
\ref{exactvectornorm} we have that $1\leqslant \|x\| \leqslant
36$.\label{exactpair}
\end{rmk}

\begin{prp}
Let $\{x_k\}_k$ be a block sequence in $\X$ and $n\inn$. Then
there exists $x$ supported by $\{x_k\}_k$ and $f\in W$ such that
$\{x,f\}$ is an $n$-exact pair.\label{exactpairexist}
\end{prp}

\begin{proof}
By Proposition \ref{propexistence} there exists  a
further normalized block sequence $\{y_k\}_k$ satisfying the assumptions of
Proposition \ref{c0spreadingmodel}. Therefore we may choose
 a strictly increasing sequence of natural numbers $\{n_k\}_k$ and
 an increasing sequence of subsets of the natural numbers $\{F_k\}_k$
satisfying the following.

\begin{itemize}

\item[(i)] $\#F_k\leqslant \min F_k$, therefore
$1\leqslant\|\sum_{i\in F_k}y_i\|\leqslant 5$, for all $k\inn$.

\item[(ii)] $\#F_{k+1} \geqslant 2^{\max\supp y_{\max F_k}}$, for
all $k\inn$.

\item[(iii)] For any $j,k\inn$ with $j<n_k$ and $f$ a functional
of type I$_\al$ in $W$ with $w(f) = j$, we have that
$|f(\sum_{i\in F_k}y_i)| < \frac{5}{2^j}$.

\end{itemize}

Setting $z_k = \sum_{i\in F_k}y_i$, by (i) and (iii) we conclude
that $\{z_k\}_k$ is $(5,\{n_k\}_k)\;\al$-RIS. By Proposition
\ref{sccexistence}, for $0<\e< \frac{1}{36\cdot 5\cdot 2^{3n}}$,
there exists  a subset of the natural numbers $G$ with $\min\supp
z_{\min G} \geqslant 8\cdot 5\cdot 2^{2n}$, $n_{\min G} > 2^{2n}$
and $\{c_k\}_{k\in G}\subset [0,1]$, such that $\sum_{k\in
G}c_k^\prime z_k$ is a $(n,\e(1-\e))$ s.c.c.

Setting $c_k = \frac{c_k^\prime}{1-c_{\max G}}$, it is
straightforward to check that $\sum_{k\in G\setminus\{\max G\}}c_k
z_k$ is a $(n,\e)$ s.c.c.

Set $x^\prime = 2^n\sum_{k\in G\setminus\{\max G\}}c_kz_k$. In
order for $x^\prime$ to be a $(5,1,n)$ exact vector, it remains to
be shown that $\|x^\prime\|\geqslant 1$.

We shall prove that for any $\eta > 0$, there exists $f_\eta$ a
functional of type I$_\al$ in $W$ with $\min\supp
x^\prime\leqslant\min\supp f_\eta$, $\max\supp
x^\prime\leqslant\max\supp f_\eta$ and $w(f_\eta) = n$, such that
$1\geqslant f_\eta(x^\prime)
> 1 - \eta$.

Observe that for $k\in G$, there exists $\al_k$ an $\al$-average
in $W$ with $s(\al_k) = \#F_k$, such that $\ran\al_k\subset\ran
z_k$ and $1\geqslant \al_k(z_k) > 1 - \eta$.

By (ii) we conclude that $\{\al_k\}_{k\in G}$ is very fast growing
and since $\ran\al_k\subset \ran z_k$, it is $\Sn$ admissible.
Therefore $f_\eta = \frac{1}{2^n}\sum_{k\in G}\al_k$ is of type
I$_\al$ in $W$ with $\min\supp x^\prime\leqslant\min\supp f_\eta$,
$\max\supp x^\prime\leqslant\max\supp f_\eta$ and $w(f_\eta) = n$.
By doing some easy calculations we conclude that it is the desired
functional, hence $\|x^\prime\| \geqslant 1$.

Moreover, for $0<\eta < 1/36, f = f_\eta$ and $x =
\frac{x^\prime}{f(x^\prime)}$, we have that $\{x,f\}$ is the
desired exact pair.

\end{proof}

\begin{dfn}
A quadruple $\{x,y,f,g\}$ is called an $n$-exact node if $\{x,f\}$
and $\{y,g\}$ are both $n$-exact pairs and $\max\supp f <
\min\supp y$.

A sequence of quadruples $\{x_k,y_k,f_k,g_k\}_{k=1}^\infty$ is
called a dependent sequence, if $\{x_k,y_k,f_k,g_k\}$ is an $n_k$
exact node for all $k\inn$, $\max\supp g_k < \min\supp x_{k+1}$
for all $k\inn$ and $\{f_k,g_k\}_{k=1}^\infty$ is a special
branch.\label{definitionexactnode}
\end{dfn}

\begin{rmks}
If $\{x,y,f,g\}$ is an $n$-exact node, then $(f + g)(x + y) = 2,
(f-g)(x-y) = 2, (f+g)(x-y) = 0, (f+g)(x) = 1, (f+g)(y) = 1$ and
$1\leqslant \|x\;\plusminus\; y\|\leqslant 72$.

If $\{x_k,y_k,f_k,g_k\}_{k=1}^\infty$ is a dependent sequence, by
the above and Proposition \ref{bplusellone}, we conclude that any
spreading model admitted by $\{x_k+y_k\}_k, \{x_k\}_k$ or
$\{y_k\}_k$, is $\ell_1$.

Moreover, for $k_0\inn$ and $k\geqslant k_0$ by Lemma
\ref{exactvectoraestimate} and the fact that $\min\supp
x_{k_0}\geqslant 8\cdot5\cdot2^{2{n_{k_0}}}$, we have that for any
very fast growing and $\Sj$-admissible sequence of $\al$-averages
$\{\al_q\}_{q=1}^d$ with $j<n_{k_0}$ and $s(\al_1) \geqslant
\min\supp x_{k_0}$, we have that
\begin{equation}
\sum_{q=1}^d|\al_q(x_k\;\plusminus\;y_k)| <
\frac{5}{2^{n_{k_0}}}\label{exactnodeaestimate}
\end{equation}

Similarly, by Lemma \ref{exactvectorbestimate}, for  any very fast
growing and $\Sj$-admissible sequence of $\be$-averages
$\{\be_q\}_{q=1}^d$ with $j<n_{k_0}-2$ and $s(\be_1) \geqslant
\min\supp x_{k_0}$, we have that
\begin{equation}
\sum_{q=1}^d|\be_q(x\;\plusminus\;y)| < \frac{5}{2^{n_{k_0}}}
\label{exactnodebestimate}
\end{equation}

\label{remarkexactsequence}
\end{rmks}

\begin{lem}
Let $\{x_k,y_k,f_k,g_k\}_{k=1}^\infty$ be a dependent sequence.
Then for every $k\inn$, if $n_k = w(f_k)$ and $n_{k+1} =
w(f_{k+1})$, the following holds.
\begin{equation}
\frac{1}{2^{n_{k+1}-3}}\max\supp y_k < \frac{1}{2^{n_k}}
\end{equation}\label{exactsequenceinftynorm}
\end{lem}

\begin{proof}
By the definition of the coding function $\sigma$, we have that
$n_{k+1} > 2^{n_k}\max\supp g_k \geqslant 2^{n_k}\max\supp y_k$.

Since $n_{k+1}\in L$, we have that $n_{k+1} > 9$. It easily
follows that $2^{n_{k+1}-3} > n_{k+1}$. Combining this with the
above, we conclude the desired result.
\end{proof}

\begin{prp} Let $Y$ be a block subspace of $\X$. Then there exist
 block sequences $\{x_k\}_k, \{y_k\}_k$ in $Y$ and $b =
\{f_k,g_k\}_{k=1}^\infty\in\mathcal{B}$, such that
$\{x_k,y_k,f_k,g_k\}_{k=1}^\infty$ is a dependent sequence.
\label{exactsequenceexist}
\end{prp}

\begin{proof}
Choose $n_1\in L_1$. By Proposition \ref{exactpairexist} there
exists  an $n_1$-exact node $\{x_1,y_1,f_1,g_1\}$ in $Y$.

Suppose that we have chosen  $n_k$-exact
nodes $\{x_k,y_k,f_k,g_k\}$ for  $k=1,\ldots,m$ such that $\{f_k,g_k\}_{k=1}^m$ is a
special sequence and $\max\supp g_k < \min\supp x_{k+1}$ for
$k=1,\ldots,m-1$.

Set $n_{m+1} = \sigma(f_1,g_1,\ldots,f_m,g_m)$. Then applying
Proposition \ref{exactpairexist} once more, there exists
 an $n_{m+1}$-exact node $\{x_{m+1},y_{m+1},f_{m+1},g_{m+1}\}$ in
$Y$, such that $\max\supp g_m < \min\supp x_{m+1}$.

The inductive construction is complete and
$\{x_k,y_k,f_k,g_k\}_{k=1}^\infty$ is a dependent sequence.

\end{proof}

An easy modification of the above proof yields the following.

\begin{cor}
If $X, Y$ are block subspaces of $\X$, then a dependent sequence
$\{x_k,y_k,f_k,g_k\}_{k=1}^\infty$ can be chosen, such that
$x_k\in X$ and $y_k\in Y$ for all $k\inn$.\label{remarkforhi}
\end{cor}

\begin{prp}
Let $\{x_k,y_k,f_k,g_k\}_{k=1}^\infty$ be a dependent sequence and
set $z_k = x_k - y_k$. Then for every $m \leqslant k_1 < \cdots <
k_m$ natural numbers and $c_1,\ldots,c_m$ real numbers, the
following holds.
\begin{equation}
\|\sum_{i=1}^mc_iu_{k_i}\|_u \leqslant \|\sum_{i=1}^mc_iz_{k_i}\|
\leqslant 146\|\sum_{i=1}^mc_iu_{k_i}\|_u
\end{equation}
where $\{u_k\}_k$ denotes the unconditional basis of Pe\l czynski (see Section \ref{section1}).
\label{universalspreadingmodel}
\end{prp}

\begin{proof}

Set $n_k = w(f_k)$ for all $k\inn$. Choose  natural numbers $m\leqslant k_1<\cdots
< k_m$ and $c_1,\ldots,c_m\subset[-1,1]$, such
that $\|\sum_{i=1}^mc_iu_{k_i}\|_u = 1$.

We first prove that $\|\sum_{i=1}^mc_iz_{k_i}\| \geqslant 1$.

Since $\min\supp z_{k_1} = \min\supp x_{k_1}\geqslant\min\supp x_m
\geqslant 40\cdot2^{2n_m} \geqslant 40\cdot2^{m} > 2m$ and
$\min\supp f_{k_1} \geqslant \min\supp x_{k_1}$, by the definition
of the norming set $W$, it follows that for every
$\la_1,\ldots,\la_m$ rational numbers such that
$\|\sum_{i=1}^m\la_iu_{k_i}^*\|_u\leqslant 1$, the functional $f =
\frac{1}{2}\sum_{i=1}^m\la_i(f_{k_i} - g_{k_i})$ is a functional
of type II$_-$ in $W$. We conclude that
\begin{equation*}
\|\sum_{i=1}^mc_iz_{k_i}\| \geqslant
\sup\big\{\sum_{i=1}^m\frac{1}{2}\la_i(f_{k_i} -
g_{k_i})(c_iz_{k_i}): \{\la_i\}_{i=1}^m\subset\mathbb{Q},\;
\|\sum_{i=1}^m\la_iu_{k_i}^*\|_u\leqslant 1\big\}
\end{equation*}

By Remark \ref{remarkexactsequence}, for $\la_1,\ldots,\la_q$ as
above, we have that $\sum_{i=1}^m\frac{1}{2}\la_i(f_{k_i} -
g_{k_i})(c_iz_{k_i}) = \sum_{i=1}^m\la_ic_i$. This yields the
following.

\begin{eqnarray*}
\|\sum_{i=1}^mc_iz_{k_i}\| &\geqslant&
\sup\big\{\sum_{i=1}^m\la_ic_i:
\{\la_i\}_{i=1}^m\subset\mathbb{Q},\;
\|\sum_{i=1}^m\la_iu_{k_i}^*\|_u\leqslant 1\big\}\\ &=&
\|\sum_{i=1}^mc_iu_{k_i}\|_u = 1
\end{eqnarray*}

To prove the inverse inequality, we will follow similar steps, as
in the proof of Proposition \ref{finitec0}. We shall inductively
prove the following.

\begin{enumerate}

\item[(i)] For any $f\in W$, we have that $|f(\sum_{k=1}^nc_kx_k)|
< 146$.

\item[(ii)] If $f$ is of type I$_\al$ or type I$_\be$ and
$w(f)\geqslant 9$, then $|f(\sum_{k=1}^nc_kx_k)| < 72 + 1/4$.

\end{enumerate}

For any functional in $W_0$ the inductive assumption holds.Assume
that it holds for any $f\in W_p$ and let $f\in W_{p+1}$. If $f$ is
a convex combination, then there is nothing to prove.

Assume that $f$ is of type I$_\be, f =
\frac{1}{2^j}\sum_{q=1}^d\be_q$, where $\{\be_q\}_{q=1}^d$ is a
very fast growing and $\mathcal{S}_j$-admissible sequence of
$\be$-averages in $W_p$.

Set $q_1 = \min\big\{q: \ran \be_q\cap\ran
z_{k_i}\neq\varnothing$\;for some\;$i\in\{1,\ldots,m\}\big\}$.

We distinguish 3 cases.\vskip3pt

\noindent {\em Case 1:} $j + 2<n_{k_1}$.

For $q>q_1$, we have that $s(\be_q) > \min\supp x_{k_1}$,
therefore, using \eqref{exactnodebestimate} we conclude that
\begin{equation}
\sum_{q>q_1}|\be_q(\sum_{i=1}^mc_iz_{k_i})| <
\frac{m}{2^{n_{k_1}}} < \frac{m}{2^m} < 1 \label{universaleq1}
\end{equation}
while the inductive assumption yields that
\begin{equation}
|\be_{q_1}(\sum_{i=1}^mc_iz_{k_i})| < 146 \label{universaleq2}
\end{equation}
Then \eqref{universaleq1} and \eqref{universaleq2} allow us to
conclude that
\begin{equation}
|f(\sum_{i=1}^mc_iz_{k_i})| < \frac{147}{2^j}\label{universaleq3}
\end{equation}\vskip3pt

\noindent {\em Case 2:} There exists $i_0 < m$, such that
$n_{k_{i_0}}\leqslant j + 2 < n_{k_{i_0} + 1}$.

Arguing as previously we get that
\begin{equation}
|f(\sum_{i>i_0}c_iz_{k_i})| < \frac{147}{2^{n_{k_{i_0}+1}}} <
\frac{147}{2^{11}} < \frac{1}{8} \label{universaleq4}
\end{equation}
and by Lemma \ref{exactsequenceinftynorm}
\begin{equation}
|f(\sum_{i<i_0}c_iz_{k_i})| < \frac{1}{2^{n_{k_1}}} < \frac{1}{8}
\label{universaleq5}
\end{equation}
Using \eqref{universaleq4}, \eqref{universaleq5} and the fact that
$|f(z_{k_{i_0}})| \leqslant 72$, we conclude that
\begin{equation}
|f(\sum_{k=1}^mc_kx_k)| < 72 + \frac{1}{4}\label{universaleq6}
\end{equation}

\noindent {\em Case 3:} $j + 2 \geqslant n_{k_m}$

By using the same arguments, we conclude that
\begin{equation}
|f(\sum_{i=1}^mc_iz_{k_i})| < 72 + \frac{1}{4}\label{universaleq7}
\end{equation}

Then \eqref{universaleq3}, \eqref{universaleq6} and
\eqref{universaleq7} yield that (i) and (ii) from the inductive
assumption are satisfied.

If $f$ is of type I$_\al$, using \eqref{exactnodeaestimate} and
the exact same arguments one can prove that (i) and (ii) from the
inductive assumption are again satisfied.

Assume now that $f$ is of type II$_-$ (or $f$ is of type II$_+$),
$f = E\big(\frac{1}{2}\sum_{j=1}^d\la_j(f_{q_j}^\prime -
g_{q_j}^\prime)\big)$ (or $f =
E\big(\frac{1}{2}\sum_{j=1}^d(f_{q_j}^\prime +
g_{q_j}^\prime)\big)$), where $E$ is an interval of the natural
numbers, $\{f_q^\prime, g_q^\prime\}_{q=1}^\infty\in\mathcal{B}$,
$q_1<\cdots<q_d$ and $2q_d \leqslant \min\supp f_{q_1}^\prime$.

We may clearly assume that $\ran(f_{q_1}^\prime \plusminus
g_{q_1}^\prime)\cap\ran(\sum_{i=1}^mc_iz_{k_i})\neq\varnothing$
and $\min E\geqslant \min\supp f_{q_1}^\prime$.

Similarly, we assume that $\ran(f_{q_d}^\prime \plusminus
g_{q_d}^\prime)\cap\ran(\sum_{i=1}^mc_iz_{k_i})\neq\varnothing$
and $\max E\leqslant \max\supp g_{q_d}^\prime$.

The inductive assumption yields  the following.
\begin{equation}
|E\big(\frac{1}{2}(f_{q_1}^\prime \plusminus
g_{q_1}^\prime)\big)(\sum_{i=1}^mc_iz_{k_i})| < 72 +
\frac{1}{4}\label{universaleq8}
\end{equation}

Set $t_j = w(f_{q_j}^\prime)$ for $j=1,\ldots,d$. By the
definition of the coding function, we have that $t_j >
2^{t_1}\min\supp x_{k_1} > \min\supp x_m > 40\cdot2^m$, for
$j=2,\ldots,d$. We conclude the following.
\begin{equation}
\sum_{j>1}\frac{72m}{2^{t_j}} \leqslant \frac{144m}{2^{t_{2}}}
<\frac{144m}{2^{40}\cdot2^m} < \frac{1}{4}\label{universaleq9}
\end{equation}

We distinguish two cases.\vskip3pt

\noindent {\em Case 1:} There exist $2\leqslant j_0\leqslant d$
and $k\inn$ such that $t_j = n_k$.\vskip3pt

In this case, the fact that $\sigma$ is one to one, yields that
$f_{q_j}^\prime \plusminus g_{q_j}^\prime = f_{q_j} \plusminus
g_{g_j}$ for $2\leqslant j < j_0$ and hence
\begin{eqnarray}
|E\big(\frac{1}{2}\sum_{j=2}^{j_0-1}\la_j(f_{q_j}^\prime -
g_{q_j}^\prime)\big)(\sum_{i=1}^mc_iz_{k_i})| &=&
|\frac{1}{2}\sum_{j=2}^{j_0-1}\la_j(f_{q_j}^\prime -
g_{q_j}^\prime)(\sum_{i=1}^mc_iz_{k_i})|\nonumber\\
&=& |\frac{1}{2}\sum_{j=2}^{j_0-1}\la_j(f_{q_j} -
g_{q_j})(\sum_{i=1}^mc_iz_{k_i})|\leqslant 1
\label{universaleq10}
\end{eqnarray}
if $f$ is of type II$_-$ and
\begin{eqnarray}
|E\big(\frac{1}{2}\sum_{j=2}^{j_0-1}(f_{q_j}^\prime +
g_{q_j}^\prime)\big)(\sum_{i=1}^mc_iz_{k_i})| &=&
|\frac{1}{2}\sum_{j=2}^{j_0-1}(f_{q_j}^\prime +
g_{q_j}^\prime)(\sum_{i=1}^mc_iz_{k_i})|\nonumber\\
&=& |\frac{1}{2}\sum_{j=2}^{j_0-1}(f_{q_j} +
g_{q_j})(\sum_{i=1}^mc_iz_{k_i})| = 0 \label{universaleq10prime}
\end{eqnarray}
if $f$ is of type II$_+$.

The inductive assumption yields that
\begin{equation}
|E\big(\frac{1}{2}(f_{q_{j_0}}^\prime -
g_{q_{j_0}}^\prime)\big)(\sum_{i=1}^mc_iz_{k_i})| < 72 +
\frac{1}{4} \label{universaleq11}
\end{equation}

Moreover, using Corollary \ref{corexactvectorroughestimate}, for
$i=1,\ldots,m$ we have that
\begin{equation*}
|E\big(\frac{1}{2}\sum_{j=j_0+1}^{d}\la_j(f_{q_j}^\prime -
g_{q_j}^\prime)\big)(z_{k_i})| < \sum_{j>j_0}\frac{72}{2^{t_j}} +
\frac{22}{2^{n_{k_i}}}
\end{equation*}
Combining this with \eqref{universaleq9}
\begin{eqnarray}
|E\big(\frac{1}{2}\sum_{j=j_0+1}^{d}\la_j(f_{q_j}^\prime -
g_{q_j}^\prime)\big)(\sum_{i=1}^mc_iz_{k_i})| &<&
\sum_{j>1}\frac{72m}{2^{t_j}} +
\sum_{i=1}^m\frac{22}{2^{n_{k_i}}}\nonumber\\
 &<& \frac{1}{4} +
\frac{22}{1000} < \frac{1}{2}\label{universaleq12}
\end{eqnarray}

Similarly,
\begin{equation}
|E\big(\frac{1}{2}\sum_{j=j_0+1}^{d}(f_{q_j}^\prime +
g_{q_j}^\prime)\big)(\sum_{i=1}^mc_iz_{k_i})| < \frac{1}{2}
\label{universaleq12prime}
\end{equation}

If $f$ is of type II$_-$ Combining \eqref{universaleq8},
\eqref{universaleq10} and \eqref{universaleq12}, we conclude that
$|f(\sum_{i=1}^mc_iz_{k_i})| < 146$, while if $f$ is of type
II$_+$ combining \eqref{universaleq8}, \eqref{universaleq10prime}
and \eqref{universaleq12prime}, we conclude that
$|f(\sum_{i=1}^mc_iz_{k_i})| < 145$\vskip3pt

\noindent {\em Case 2:} $t_j\neq n_k$, for all $j=2,\ldots,d$ and
$k\inn$.\vskip3pt

Arguing as previously, we conclude that

\begin{eqnarray}
|E\big(\frac{1}{2}\sum_{j=2}^{d}\la_j(f_{q_j}^\prime -
g_{q_j}^\prime)\big)(\sum_{i=1}^mc_iz_{k_i})|  <
\frac{1}{2}\quad\text{and}\\|E\big(\frac{1}{2}\sum_{j=2}^{d}(f_{q_j}^\prime
+ g_{q_j}^\prime)\big)(\sum_{i=1}^mc_iz_{k_i})|  <
\frac{1}{2}\nonumber\label{universaleq13}
\end{eqnarray}

Therefore, \eqref{universaleq8} and \eqref{universaleq13} yield
that $|f(x)| < 73$. The induction is complete and so is the proof.

\end{proof}

\begin{prp}
Let $Y$ be a block subspace of $\X$. Then there exist 
a seminormalized block sequence $\{z_k\}_k$ in $Y$ and  a
seminormalized block sequence $\{z_k^*\}_k$ in $\X^*$ satisfying the following.

\begin{itemize}

\item[(i)] $z^*_k(z_n) = \de_{k,n}$

\item[(ii)] For every suppression unconditional and spreading sequence
$\{w_n\}_n$, there exists $\{k_n\}_n$ a strictly increasing
sequence of natural numbers, such that $\{z_{k_n}\}_n$ generates a
spreading model which is 146-equivalent to $\{w_n\}_n$ and
$\{z_{k_n}^*\}_n$ generates a spreading model which is
146-equivalent to $\{w_n^*\}_n$

\end{itemize}
\label{spreadingmodeluniversal}
\end{prp}

\begin{proof}
By Proposition \ref{exactsequenceexist}, there exists
 a dependent sequence $\{x_k,y_k,f_k,g_k\}_{k=1}^\infty$ in $Y$.
Set $z_k = x_k - y_k$ and $z_k^* = \frac{1}{2}(f_k - g_k)$. Then
$z^*_k(z_n) = \de_{k,n}$.

Let $\{w_n\}_n$ be a suppression unconditional and spreading sequence, then there
exists  a strictly increasing sequence of natural
numbers $\{k_n\}_n$, such that $\{u_{k_n}\}_{n\geqslant j}$ is $1+\e_j$
equivalent to $\{w_n\}_{n\geqslant j}$, where $\{\e_j\}_j$ is a null
sequence of positive reals.

Moreover, due to unconditionality, $\{u_{k_n}^*\}_{n\geqslant j}$
is $1+\e_j$ equivalent to $\{w_n^*\}_{n\geqslant j}$.

Proposition \ref{universalspreadingmodel} yields that for every
 natural numbers $m\leqslant n_1 <\cdots< n_m$ and 
real numbers $c_1,\cdots,c_m$, we have that
\begin{equation}
\frac{1}{1+\e_m}\|\sum_{i=1}^mc_iw_i\| \leqslant
\|\sum_{i=1}^mc_iz_{n_{k_i}}\| \leqslant
(1+\e_m)146\|\sum_{i=1}^mc_iw_i\|\label{usmequation1}
\end{equation}
This yields that any spreading model admitted by $\{z_{k_n}\}_n$
is 146-equivalent to $\{w_n\}_n$.

Moreover, by the definition of the norming set, for every
 natural numbers $m\leqslant n_1 <\cdots< n_m$ and 
real numbers $c_1,\cdots,c_m$, we have that
\begin{equation}
\|\sum_{i=1}^mc_iz_{n_{k_i}}^*\| \leqslant
\|\sum_{i=1}^mc_iu_{n_{k_i}}^*\|_u \leqslant
(1+\e_m)\|\sum_{i=1}^mc_iw_i^*\|\label{usmequation2}
\end{equation}
Property (i) and \eqref{usmequation1} yield the following.
\begin{equation}
\frac{1}{146(1+\e_m)}\|\sum_{i=1}^mc_iw_i^*\| \leqslant
\|\sum_{i=1}^mc_iz_{n_{k_i}}^*\|\label{usmequation3}
\end{equation}
Combining \eqref{usmequation2} and \eqref{usmequation3}, we
conclude that any spreading model admitted by $\{z_{k_n}^*\}_n$ is
146-equivalent to $\{w_n^*\}_n$.
\end{proof}

\begin{prp}
The space $\X$ is hereditarily indecomposable.
\end{prp}

\begin{proof}
It is enough to show that for  block subspaces $X,Y$ of $\X$ and
$\e >0$, there exist $x\in X$ and $y\in Y$ such that $\|x + y\|
\geqslant 1$ and $\|x - y\| < \e$.

By Corollary \ref{remarkforhi}, there exists
 a dependent sequence $\{x_k,y_k,f_k,g_k\}_{k=1}^\infty$, with
$x_k\in X$ and $y_k\in Y$ for all $k\inn$.

By Remark \ref{remarkexactsequence} and Proposition
\ref{universalspreadingmodel}, there exists  a strictly
increasing sequence of natural numbers $\{k_n\}_n$, such that $\{x_{k_n} +
y_{k_n}\}_n$ generates an $\ell_1$ spreading model and $\{x_{k_n}
- y_{k_n}\}_n$ generates a $c_0$ spreading model.

Fix $c>0$ such that for any 
natural numbers $m\leqslant n_1 < \cdots < n_m$ the following holds.
\begin{eqnarray*}
\frac{1}{m}\|\sum_{i=1}^m(x_{k_{n_i}} - y_{k_{n_i}})\| \leqslant
\frac{1}{c\cdot m}\\
\frac{1}{m}\|\sum_{i=1}^m(x_{k_{n_i}} + y_{k_{n_i}})\| \geqslant c
\end{eqnarray*}

Fix  natural numbers $m \leqslant n_1 < \cdots < n_m$ such that
$\frac{1}{c^2m} < \e$ and set $x = \frac{1}{c\cdot
m}\sum_{i=1}^mx_{k_{n_i}}$ and $y = \frac{1}{c\cdot
m}\sum_{i=1}^my_{k_{n_i}}$.

Then $\|x + y\| \geqslant 1$ and $\|x - y\| \leqslant
\frac{1}{c^2m} < \e$.

\end{proof}

\section{Bounded operators on $\X$}

This section is devoted to operators on $\X$. We prove that in
every block subspace of $\X$ there exist equivalent intertwined
block sequences $\{x_k\}_k, \{y_k\}_k$ and an onto isomorphism
$T:\X\rightarrow\X$, such that $Tx_k = y_k$. This yields that $\X$
does not contain a block subspace that is tight by range and
hence, $\X$ is saturated with sequentially minimal subspaces (see
\cite{FR}). We then proceed to identify block sequences witnessing
this fact and we also prove that the whole space $\X$ is sequencially minimal. We moreover construct a strictly singular operator
$S:\X\rightarrow\X$ which is not polynomially compact. All the
above properties of $\X$ are based on the way type II functionals
are constructed in the norming set $W$ and the rich spreading
model structure of $\X$.

The following result is proven in a similar manner as Theorem 5.8
from \cite{AM} and therefore its proof is omitted.

\begin{prp} Let $Y$ be an infinite dimensional closed subspace of
$\X$ and $T:Y\rightarrow \X$ be a bounded linear operator. Then
there exists $\la\in\mathbb{R}$, such that $T - \la I_{_{Y,\X}}: Y
\rightarrow \X$ is strictly singular.\label{scalarplusss}
\end{prp}

The following result follows from Proposition 3.1 from \cite{ADT},
see also \cite{MP}.

\begin{prp}
Let $\{x_m^*\}_m$ be a block sequence in $\X^*$ generating a $c_0$
spreading model and $\{x_k\}_k$ be a block sequence in $\X$
generating a spreading model which is not equivalent to $\ell_1$.
Then there exists a strictly increasing sequence of natural
numbers $\{t_j\}_j$, such that the following is satisfied. For
every strictly increasing sequence of natural numbers $\{m_k\}_k$
with $m_k\geqslant t_k$ for all $k\inn$, the map $T:\X\rightarrow
\X$ with $Tx = \sum_{k=1}^\infty x_{m_k}^*(x)x_k$ is bounded and
non compact. \label{operator}
\end{prp}

The proof of the following result uses an argument, which first
appeared in \cite{FS}, namely the following. If $\{x_k\}_k$,
$\{y_k\}_k$ are basic sequences in a space $X$, such that the maps
$x_k\rightarrow x_k - y_k$ and $y_k\rightarrow x_k - y_k$ extend
to bounded linear operators, then $\{x_k\}_k$ is equivalent to
$\{y_k\}_k$.

\begin{prp}
Let $\{x_k,y_k,f_k,g_k\}_{k=1}^\infty$ be a dependent sequence.
Then there exists  a strictly increasing sequence of
natural numbers $\{k_n\}_n$, such that $\{x_{k_n}\}_n$ is equivalent to
$\{y_{k_n}\}_n$. More precisely, there exists 
an onto isomorphism $T:\X\rightarrow \X$, with $Tx_{k_n} = y_{k_n}$ for all
$n\inn$.\label{exactsequenceevenodd}
\end{prp}

\begin{proof}
First observe the following, for any $k\inn$, we have that
\begin{equation*}
2\geqslant \|f_k + g_k\| \geqslant
(f_k+g_k)\big(\frac{x_k+y_k}{\|x_k + y_k\|}\big)\geqslant
\frac{2}{72}
\end{equation*}
Hence $\{f_k + g_k\}_k$ is seminormalized and by the definition of
the norming set $W$, any spreading model admitted by it, is $c_0$.

By Proposition \ref{universalspreadingmodel}, $\{x_k-y_k\}_k$
admits a $c_0$ spreading model. Proposition \ref{operator}, yields
that there exists  a strictly increasing sequence of
natural numbers $\{k_n\}_n$, such that the operator $S:\X\rightarrow\X$ with
\begin{equation*}
Sx = \sum_{n=1}^\infty(f_{k_n} + g_{k_n})(x)(x_{k_n} - y_{k_n})
\end{equation*}
is bounded.

Then, for every $n\inn$ we have that $Sx_{k_n} = x_{k_n} -
y_{k_n}$. Setting $T = I - S$, we evidently have that $Tx_{k_n} =
y_{k_n}$, hence $\{x_k\}_k$ is dominated by $\{y_k\}_k$.

Similarly, for every $n\inn$ we have that $Sy_{k_n} = x_{k_n} -
y_{k_n}$. Setting $Q = I + S$, we evidently have that $Qy_{k_n} =
x_{k_n}$. Therefore $\{y_k\}_k$ is dominated by $\{x_k\}_k$, which
yields that they are actually equivalent.

We shall moreover prove that $T$ is invertible, in fact $Q =
T^{-1}$. Notice that $TQ = QT = I - S^2$. It remains to be shown
that $S^2 = 0$.

Since $Sx_{k_n} = x_{k_n} - y_{k_n} = Sy_{k_n}$ for all $n\inn$,
we evidently have that $S(x_{k_n} - y_{k_n}) = 0$ for all $n\inn$.
This yields that $[\{x_{k_n} - y_{k_n}\}_n] \subset \ker S$.
Evidently, we have that $S[\X]\subset [\{x_{k_n} - y_{k_n}\}_n]$,
therefore $S[\X]\subset \ker S$. We conclude that $S^2 = 0$ and
this completes the proof.
\end{proof}

Before the statement of the next result, we remind the notion of
even-odd sequences and intertwined block sequences. A Schauder
basic sequence $\{x_k\}_k$ is called even-odd, if $\{x_{2k}\}_k$
is equivalent to $\{x_{2k-1}\}_k$  (see \cite{G}).

Two block sequences $\{x_k\}_k$, $\{y_k\}_k$ are called
intertwined, if $x_k < y_k < x_{k+1}$ for all $k\inn$.

Evidently, two intertwined block sequences $\{x_k\}_k$,
$\{y_k\}_k$ are equivalent, if and only if the sequence
$\{z_k\}_k$ with $z_{2k-1} = x_k$ and $z_{2k} = y_k$ for all
$k\inn$, is an even-odd sequence.

\begin{prp}
Every block subspace of $\X$ contains an even-odd block sequence.
More precisely, in every block subspace $Y$ of $\X$, there exists
a block sequence $\{z_k\}_k$ and  an onto
isomorphism $T:\X\rightarrow \X$, such that $Tz_{2k-1} = z_{2k}$, for all
$k\inn$.\label{evenodd}
\end{prp}

\begin{proof}
By Proposition \ref{exactsequenceexist}, there exists
 a dependent sequence $\{x_k,y_k,f_k,g_k\}_{k=1}^\infty$ in $Y$ and
by Proposition \ref{exactsequenceevenodd} there exist 
a strictly increasing sequence of natural numbers $\{k_n\}_n$ and
 an onto isomorphism $T:\X\rightarrow\X$, such that $Tx_{n_k} =
y_{n_k}$ for all $k\inn$. Setting $z_{2k-1} = x_{n_k}$ and $z_{2k}
= y_{n_k}$ for all $k\inn$, we have that $\{z_k\}_k$ is the
desired even-odd block sequence and $T$ the desired operator.
\end{proof}

\begin{cor}
The space $\X$ does not contain a block subspace which is tight by
range.\label{tightbyrange}
\end{cor}

Theorem 1.4 from \cite{FR} yields that $\X$ is saturated with
sequentially minimal block subspaces. The next result identifies
block subspaces of $\X$ with the aforementioned property and also implies the sequential minimality of the whole space $\X$.

\begin{prp}
There exists a set of block sequences $\big\{\{x_k^{(Y)}\}_k: Y$\;
is a block subspace of $\X\big\}$, with $\{x_k^{(Y)}\}_k\subset Y$
for every $Y$ block subspace of $\X$, satisfying the following.
For every  block subspaces $Y, Z$ of $\X$, there exist  strictly increasing sequences of natural numbers $\{k_n\}_n,
\{m_n\}_n$, such
that $\{x_{k_n}^{(Y)}\}_n$ and $\{x_{m_n}^{(Z)}\}_n$ are
intertwined and equivalent. More precisely, there exists
 an onto isomorphism $T:\X\rightarrow \X$, such that
$Tx_{k_n}^{(Y)} = x_{m_n}^{(Z)}$ for all $n\inn$.\label{seqmin}
\end{prp}

\begin{proof}
Let $Y$ be a block subspace of $\X$. By Proposition
\ref{exactpairexist}, we may choose a block sequence $\{x_k\}_k$
in $Y$, satisfying the following.
\begin{itemize}

\item[(i)]There exists  a sequence of type I$_\al$
functionals $\{f_k\}_k$ in $W$, such that $\{x_k,f_k\}$ is a $w(f_k)$-exact
pair for all $k\inn$.

\item[(ii)] For every $n\inn$, the set $\{k\inn: w(f_k) = n\}$ is
infinite.
\end{itemize}

For every  block subspace  $Y$  of $\X$, choose $\{x_k^{(Y)}\}_k$
satisfying properties (i) and (ii).

Let now $Y,Z$ be block subspaces of $\X$. We shall recursively
choose  strictly increasing sequences of
natural numbers $\{k_n\}_n, \{m_n\}_n$ and  sequences of type
I$_\al$ functionals $\{f_n\}_n, \{g_n\}_n$, such that
$\{x_{k_n}^{(Y)},x_{m_n}^{(Z)},f_n,g_n\}_{n=1}^\infty$ is an exact
sequence.

Choose $p_1\in L_1$ and $k_1\inn, f_1\in W$ a functional of type
I$_\al$, such that $\{x_{k_1}^{(Y)}, f_1\}$ is a $p_1$ exact pair.

Similarly, choose $m_1\inn, g_1\in W$ a functional of type
I$_\al$, such that $\{x_{m_1}^{(z)}, f_1\}$ is a $p_1$ exact pair
and $\max\supp f_1 < \min\supp x_{m_1}^{(z)}$.

Suppose that we have chosen  strictly increasing sequences of natural
numbers $\{k_n\}_{n=1}^\ell,
\{m_n\}_{n=1}^\ell$ and sequences of
type I$_\al$ functionals $\{f_n\}_{n=1}^\ell, \{g_n\}_{n=1}^\ell$, such that
$\{x_{k_n}^{(Y)},x_{m_n}^{(Z)},f_n,g_n\}$ are $p_n$-exact nodes
for $k=1,\ldots,\ell$ $n_\ell$, $\{f_n,g_n\}_{n=1}^\ell$ is a
special sequence and $\max\supp g_n < \min\supp x_{n+1}^{(Y)}$ for
$k=1,\ldots,m-1$.

Set $p_{\ell+1} = \sigma(f_1,g_1,\ldots,f_\ell,g_\ell)$. Then
arguing as previously, we may choose $k_{\ell+1} > k_\ell,
m_{\ell+1}>m_\ell$ and  functionals of
type I$_\al$ $f_{\ell+1}, g_{\ell+1}$, such that
$\{x_{k_{\ell+1}}^{(Y)},x_{m_{\ell+1}}^{(Z)},f_{\ell+1},g_{\ell+1}\}$
is an $p_{\ell+1}$-exact node and $\max\supp g_\ell < \min\supp
x_{m_{\ell+1}}^{(Y)}$.

The inductive construction is complete and
$\{x_{k_n}^{(Y)},x_{m_n}^{(Z)},f_n,g_n\}_{n=1}^\infty$ is a
dependent sequence.

Proposition \ref{exactsequenceevenodd} yields the desired result.
\end{proof}

A related result to the following can be found in \cite{MP},
Proposition 2.1.

\begin{prp}
Let $1<q<\infty$, $q^\prime$ be its conjugate and set $t_j =
\lceil (4\cdot 2^{j+1})^{q^\prime}\rceil$. Then the following
holds.

If $\{m_j\}_j$ is a strictly increasing sequence of natural
numbers with $m_j \geqslant t_j$ for all $j\inn$, $\{x_m^*\}_m$ is
a block sequence in $\X^*$ and $\{x_k\}_k$ is a block sequence in
$\X$ satisfying the following,

\begin{itemize}

\item[(i)] $\{x_m^*\}_m$ is either generating an $\ell_p$
spreading model, with $p> q^\prime$, or a $c_0$ spreading model

\item[(ii)] $\{x_k\}_k$ is either generating an $\ell_r$ spreading
model with $r\geqslant q$, or a $c_0$ spreading model

\end{itemize}
then the map $T:\X\rightarrow \X$ with $Tx = \sum_{k=1}^\infty
x_{m_k}^*(x)x_k$ is bounded and non compact.

If moreover $\dim(\X/[\{x_k\}_k]) = \infty$, then $T$ is strictly
singular.\label{operatorpq}
\end{prp}

\begin{proof}
If $\{x_m^*\}_m$ generates a $c_0$ spreading model, fix $q^\prime
< p < \infty$. Note that by the choice of $t_j$, we have that
\begin{eqnarray*}
\frac{t_j^{1/p}}{2^j}&\leqslant& \frac{\big((4\cdot
2^{j+1})^{q^\prime}+1\big)^{1/p}}{2^j} \leqslant \frac{(4\cdot
2^{j+1})^{q^\prime/p}}{2^j} + \frac{1}{2^j}\\
&=& 8^{q^\prime/p}\frac{1}{(2^{1-q^\prime/p})^j} + \frac{1}{2^j}
\end{eqnarray*}
Since $p>q^\prime$, we have that
$\sum_{j=1}^\infty\frac{1}{(2^{1-q^\prime/p})^j}<\infty$. We
conclude that if we set
\begin{equation*}
\al =
8^{q^\prime/p}\sum_{j=1}^\infty\frac{1}{(2^{1-q^\prime/p})^j} + 1
\end{equation*}
Then
\begin{equation}
\sum_{j=1}^\infty\frac{t_j^{1/p}}{2^j} \leqslant
\al\label{operatoreq1}
\end{equation}

Fix $C>0$ such that for any  natural
numbers $n\leqslant m_1 < \cdots <m_n$ and  real numbers $c_1,\ldots,c_m$ the following holds.
\begin{equation}
\|\sum_{i=1}^n c_ix_{m_i}^*\| \leqslant
C(\sum_{i=1}^n|c_i|^p)^{1/p}\label{operatoreq2}
\end{equation}

By multiplying the $x_k$ with an appropriate scalar, we may assume
that $\|x_k\| \leqslant 1/2$ for all $k\inn$ and that for any
 natural numbers $n\leqslant m_1 < \cdots <m_n$ and
 real numbers $c_1,\ldots,c_m$ the following holds.
\begin{equation}
\|\sum_{i=1}^n c_ix_{m_i}\| \leqslant
(\sum_{i=1}^n|c_i|^q)^{1/q}\label{operatoreq3}
\end{equation}

Let $x\in X, \|x\| = 1$, $x^*\in Y^*, \|x^*\| = 1$. For $j\inn$,
set
\begin{equation*}
B_j = \{k\inn: \frac{1}{2^{j+1}}<|x^*(x_k)|\leqslant
\frac{1}{2^j}\}
\end{equation*}

Then $\{B_j\}_j$ is a partition of the natural numbers and
\begin{equation}
|x^*(Tx)| \leqslant \sum_{j=1}^\infty|\sum_{k\in
B_j}x^*(x)x^*_{m_k}(x)|\label{operatoreq4}
\end{equation}
We will show that $\#B_j \leqslant t_j$.

Assume that this is not the case. Then we may choose $F\subset
B_j$ with $\#F > t_j/2$ and $\#F\leqslant \min F$.

Set
\begin{eqnarray*}
F_1 &=& \{k\in B_j: x^*(x_k) \geqslant 0\}\\
F_2 &=& \{k\in B_j: x^*(x_k) < 0\}
\end{eqnarray*}
Then either $\#F_1 > t_j/4$, or $\#F_2 > t_j/4$ and we shall
assume the first. Choose $G\subset F_1$ with $\#G = \lceil
t_j/4\rceil$.

Then, by \eqref{operatoreq3} and the choice of $G$, we have the
following.

\begin{equation*}
t_j^{1/q} \geqslant \|\sum_{k\in G}x_k\| \geqslant x^*(\sum_{k\in
G}x_k)
> \frac{t_j}{4\cdot2^{j+1}}
\end{equation*}
We conclude that $t_j < (4\cdot 2^{j+1})^{q^\prime}$, which
contradicts the choice of $t_j$.

Set
\begin{equation*}
C_j = \{k\in B_j: k\geqslant j\},\quad D_j = B_j\setminus C_j
\end{equation*}
Evidently $\#D_j\leqslant j-1$, hence
\begin{equation}
|\sum_{k\in D_j}x^*(x_k)x^*_{m_k}(x)| \leqslant
\frac{j-1}{2^j}\label{operatoreq5}
\end{equation}
Moreover,
\begin{equation*}
\#\{m_k: k\in C_j\}\leqslant t_j \leqslant \min\{t_k: k\in C_j\}
\leqslant \min\{m_k: k\in C_j\}
\end{equation*}
Therefore, using \eqref{operatoreq2} and the definition of $C_j$,

\begin{eqnarray*}
|\sum_{k\in C_j}x^*(x_k)x^*_{m_k}(x)| &\leqslant& \|\sum_{k\in
C_j}x^*(x_k)x^*_{m_k}\|\\
 &\leqslant& C(\sum_{k\in C_j}|x^*(x_k)|^p)^{1/p} \leqslant
 \frac{C\cdot t_j^{1/p}}{2^j}
\end{eqnarray*}

The above, combined with \eqref{operatoreq1}, \eqref{operatoreq4}
and \eqref{operatoreq5} yields the following.
\begin{eqnarray*}
|x^*(Tx)| &\leqslant& \sum_{j=1}^\infty|\sum_{k\in
B_j}x^*(x)x^*_{m_k}(x)|\\
&\leqslant& \sum_{j=1}^\infty|\sum_{k\in C_j}x^*(x_k)x^*_{m_k}(x)|
+ \sum_{j=1}^\infty|\sum_{k\in D_j}x^*(x_k)x^*_{m_k}(x)|\\
&\leqslant& C\sum_{j=1}^\infty\frac{t_j^{1/p}}{2^j} +
\sum_{j=1}^\infty\frac{j-1}{2^j}\\
&\leqslant& C\cdot\al + 1
\end{eqnarray*}

We conclude that $\|T\| \leqslant C\cdot\al + 1$. The non
compactness of $T$ follows easily, if we consider $\{z_k\}_k$ the
biorthogonals of $\{x_{m_k}^*\}_k$. Then $\{z_k\}_k$ is
seminormalized and $\{Tz_k\}_k = \{x_k\}_k$, therefore it is not
norm convergent.

We now prove that $T$ is strictly singular. Suppose that it is
not, then by Proposition \ref{scalarplusss}, there exists $\la\neq
0$ such that $Q = T - \la I$ is strictly singular. Since $\la I$
is a Fredholm operator and $Q$ is strictly singular, it follows
that $T = Q + \la I$ is also a Fredholm operator, therefore
$\dim(\X/T[\X]) < \infty$. The fact that $T[\X] \subset
[\{x_k\}_k]$ and $\dim(\X/[\{x_k\}_k]) = \infty$ yields a
contradiction.

\end{proof}

\begin{prp}
There exists $S:\X\rightarrow \X$ a strictly singular operator
which is not polynomially compact.\label{ssnonpolynomiallycompact}
\end{prp}

\begin{proof}
Choose  a strictly increasing sequence of real numbers $\{p_n\}_n$,
with $p_1 > 2$ and let $p_n^\prime$ be the conjugate of $p_n$ for
all $n\inn$.

By Proposition \ref{spreadingmodeluniversal}, for every $n\inn$
there exist  a seminormalized block sequence $\{x_k^n\}_k$ in $\X$,
with $\|x_k^n\|\geqslant 1$ for all $k,n\inn$ and 
a seminormalized block sequence $\{x_k^{n*}\}_k$ in $\X^*$, satisfying the
following.
\begin{itemize}

\item[(i)] $x^{n*}_k(x_m^n) = \de_{k,m}$

\item[(ii)] $\{x_k^n\}_k$ generates an $\ell_{p_n}$ spreading
model and $\{x_k^{n*}\}_k$ generates an $\ell_{p_n^\prime}$
spreading model.

\end{itemize}

If we set $E_k^n = \ran(\ran x_k^n\cup\ran x_k^{n*})$, using a
diagonal argument we may assume that the intervals
$\{E_k^n\}_{k,n}$ are pairwise disjoint.

Set $m_k = \lceil(4\cdot 2^{k+1})^2\rceil$ and
$S_n:\X\rightarrow\X$ with
\begin{equation*}
S_nx = \sum_{k=1}^\infty x_{m_k}^{n*}(x)x_{m_k}^{n+1}
\end{equation*}
Proposition \ref{operatorpq} (for $q=p_{n+1}$), yields that $S_n$
is bounded and strictly singular. Moreover the following holds.
\begin{itemize}

\item[(a)]For every $k,n\inn$, $S_nx_{m_k}^n = x_{m_k}^{n+1}$

\item[(b)] For every $n\neq l\inn$ and $k\inn$, $S_nx_{m_k}^l =
0$.

\end{itemize}

Set $S = \sum_{n=1}^\infty\frac{1}{2^n\|S_n\|}S_n$. Then $S$ is
strictly singular and we shall prove that that it is not
polynomially compact.

Properties (a) and (b), yield that for every $k,n\inn$ we have
that $Sx_{m_k}^n = \frac{1}{2^n\|S_n\|}x_{m_k}^{n+1}$.

Using an easy induction we conclude the following.
\begin{equation}
S^nx_{m_k}^1 =
\big(\prod_{j=1}^n\frac{1}{2^j\|S_j\|}\big)x_{m_k}^{n+1},\quad\text{for
every}\;k,n\inn\label{ssnonpolynomiallycompactequation1}
\end{equation}

Set $a_n = \prod_{j=1}^n\frac{1}{2^j\|S_j\|}$ for $n\inn$ and $a_0
= 1$.

Let now $T = \sum_{n=0}^db_nS^n$ be a non zero polynomial of $S$.
Then, using \eqref{ssnonpolynomiallycompactequation1}, for every
$k\inn$, we have that
\begin{equation*}
Tx^1_{m_k} = \sum_{n=0}^db_na_nx_{m_k}^{n+1}
\end{equation*}
The fact that the basis of $\X$ is bimonotone, the
$x_{m_k}^1,\ldots,x_{m_k}^{d+1}$ are disjointly ranged and
$\|x_{m_k}^n\|\geqslant 1$, for all $k,n\inn$, yields that
$\|Tx^1_{m_k}\| \geqslant \max\{|a_nb_n|: n=0,\ldots,d\}$, for all
$k\inn$. We conclude that $\{Tx^1_{m_k}\}_k$ has no norm
convergent subsequence, therefore $T$ is not compact.
\end{proof}

\begin{rmk}
A slight modification of the above yields that in every block
subspace of $\X$ there exists a strictly singular operator which
is not polynomially compact.
\end{rmk}

We close the paper with the following two problems, which are open
to us.

\begin{prb}
Does there exist a reflexive Banach space with an unconditional
basis, which is hereditarily unconditional spreading model
universal?
\end{prb}

Although it does not seem necessary to use conditional structure
in order to construct a hereditarily unconditional spreading model
universal space, in our approach the conditional structure of the
type II$_+$ functionals cannot be avoided, resulting in an HI
space.

\begin{prb}
Does there exist a Banach space hereditarily spreading model
universal, for both conditional and unconditional spreading
sequences?
\end{prb}

\end{document}